\RequirePackage{fix-cm}
\documentclass[smallextended]{svjour3}       
\smartqed  
\usepackage{graphicx}                
\usepackage{color,xcolor}            
\usepackage{hyperref}
\usepackage{enumerate}

\usepackage{float}
\usepackage{graphicx}
\usepackage{amsmath}
\usepackage{mathrsfs}                
\usepackage{amsfonts}
\usepackage{amsbsy}
\usepackage{amssymb}   
\usepackage{graphics}
\usepackage{color}		
\usepackage{epsfig}
\usepackage{mathrsfs}
\usepackage{bm}                      
\usepackage{bbding,manfnt}           
\usepackage{sidecap}
\usepackage{caption}
\usepackage{subcaption}
\usepackage{cases}
\usepackage{esint}
\usepackage{enumerate}
\usepackage{lettrine}                
\usepackage[top=1in,bottom=1in,left=1.2in,right=1.2in]{geometry}

\newcommand*{\R}{\ensuremath{\mathbb{R}}}
\renewcommand*{\L}{\ensuremath{\mathbb{L}}}
\newcommand{\ms}{\ensuremath{\mathcal{S}}}
\newcommand{\mr}{\ensuremath{\mathcal{R}}}
\newcommand{\mfr}{\ensuremath{\mathfrak{R}}}
\newcommand{\mz}{\ensuremath{\mathcal{Z}}}

\newcommand*{\CC}{\ensuremath{\mathbf{C}}}
\newcommand*{\Z}{\ensuremath{\mathbb{Z}}}
\newcommand*{\T}{\ensuremath{\mathscr{T}}}
\newcommand*{\md}{\ensuremath{\mathscr{D}}}

\newcommand*{\F}{\ensuremath{\mathcal{F}}}

\newcommand*{\A}{\ensuremath{\mathcal{A}}}
\newcommand*{\oo}{\ensuremath{C}}
\newcommand*{\mfl}{\ensuremath{\mathfrak{L}}}
\newcommand*{\mflt}{\ensuremath{\mathfrak{L}_t}}

\newcommand*{\D}{\ensuremath{\mathbb{D}}}

\newcommand{\C}[1]{\mathbf{C^{#1}}}
\newcommand{\lnorm}[1]{\ensuremath{\mathbb{L}}^{#1}}
\newcommand{\llnorm}[1]{\ensuremath{\left\|#1\right\|_{\mathbb{L}^1}}}

\newcommand{\brz}{B_r(0)}

\renewcommand{\a}{\ensuremath{\alpha}}
\renewcommand{\b}{\ensuremath{\beta}}
\newcommand{\la}{\ensuremath{\lambda}}

\newcommand{\ga}{\gamma}
\newcommand{\om}{\omega}
\newcommand{\mmz}{\ensuremath{\mathit{z}}}

\newcommand{\Vl}{\ensuremath{V^{L}}}
\newcommand{\Ql}{\ensuremath{Q^{L}}}
\newcommand{\upl}{\ensuremath{\Upsilon^{L}}}
\newcommand{\Vr}{\ensuremath{V^{R}}}
\newcommand{\Qr}{\ensuremath{Q^{R}}}
\newcommand{\upr}{\ensuremath{\Upsilon^{R}}}

\newcommand{\uh}{\ensuremath{u^h}}
\newcommand{\uhi}{\ensuremath{u^{h_i}}}
\newcommand{\ueh}{\ensuremath{u^{\e,h}}}
\newcommand{\uehn}{\ensuremath{u^{\e_{\nu},h}}}
\newcommand{\veh}{\ensuremath{v^{\e,h}}}
\newcommand{\uhe}{\ensuremath{u^{h,\e}}}
\newcommand{\he}{\ensuremath{(h,\e)}}
\newcommand{\eh}{\ensuremath{(\e,h)}}
\newcommand{\bu}{\ensuremath{\bar u}}
\newcommand{\ums}{\ensuremath{u^{-}}}
\newcommand{\ups}{\ensuremath{u^+}}
\newcommand{\vms}{\ensuremath{v^{-}}}
\newcommand{\vps}{\ensuremath{v^+}}
\newcommand{\vmps}{\ensuremath{v^\pm}}
\newcommand{\umps}{\ensuremath{u^\pm}}

\newcommand*{\ul}{\ensuremath{u^\mathrm{L}}}
\newcommand*{\um}{\ensuremath{u^\mathrm{M}}}
\newcommand*{\ur}{\ensuremath{u^\mathrm{R}}}

\newcommand*{\unuh}{\ensuremath{u^{\nu,h}}}
\newcommand*{\unuoh}{\ensuremath{u^{\nu_1,h}}}
\newcommand*{\unuth}{\ensuremath{u^{\nu_2,h}}}

\newcommand*{\e}{\ensuremath{\varepsilon}}

\newcommand{\ue}{\ensuremath{u^{\varepsilon}}}

\newcommand{\vhe}{\ensuremath{v^{h,\varepsilon}}}
\newcommand*{\iue}{\ensuremath{\bar u^{\varepsilon}}}
\newcommand*{\iuhe}{\ensuremath{\bar u^{h,\varepsilon}}}
\newcommand*{\gohe}{\ensuremath{ g_1^{h,\varepsilon}}}
\newcommand*{\gthe}{\ensuremath{ g_2^{h,\varepsilon}}}

\newcommand*{\iu}{\ensuremath{\bar u}}
\newcommand*{\iv}{\ensuremath{\bar v}}

\newcommand{\enu}{\varepsilon_{\nu}}

\newcommand{\tu}{\tilde u}

\newcommand{\bxa}{\ensuremath{\bar{x}_{\a}}}
\newcommand{\ka}{\ensuremath{k_{\alpha}}}

\newcommand{\tum}{\ensuremath{\tilde{u}^M}}
\newcommand{\tur}{\ensuremath{\tilde{u}^R}}

\newcommand*{\go}{\ensuremath{g_1}}
\newcommand*{\gt}{\ensuremath{g_2}}
\newcommand*{\goe}{\ensuremath{g_1^{\varepsilon}}}
\newcommand*{\gte}{\ensuremath{g_2^{\varepsilon}}}

\newcommand*{\got}{\ensuremath{g_1(t)}}
\newcommand*{\gtt}{\ensuremath{g_2(t)}}
\newcommand{\geo}{\ensuremath{g^{\varepsilon}_1}}

\newcommand{\get}{\ensuremath{g^{\varepsilon}_2}}

\newcommand{\gie}{\ensuremath{g^{\varepsilon}_i}}

\newcommand*{\ith}{\ensuremath{i\text{-th}}}

\renewcommand*{\P}{\mathcal{P}}
\newcommand*{\NP}{\mathcal{N}\mathcal{P}}
\newcommand{\xat}{\ensuremath{x_\alpha(t)}}
\newcommand{\sa}{\ensuremath{\sigma_{\alpha}}}
\renewcommand{\sb}{\ensuremath{\sigma_{\beta}}}
\newcommand{\lug}{\Lambda(\iu,\go,\gt)}

\newcommand{\hatt}{\ensuremath{\hat{\tau}}}
\newcommand*{\yo}{\ensuremath{y_1}}
\newcommand{\tqi}{\ensuremath{\tilde{q}_i}}
\newcommand{\xa}{x_{\alpha}}
\newcommand{\xb}{x_{\beta}}

\newcommand{\dbxa}{\dot{\bar{x}}_{\alpha}}

\newcommand{\qbpi}{{q}^{b}_i(t)}
\newcommand{\tqapi}{\tilde{q}_i(t,\bar{x}_{\alpha}(t)+)}
\newcommand{\tqami}{\tilde{q}_i(t,\bar{x}_{\alpha}(t)-)}
\newcommand{\tqbpi}{\tilde{q}^{b}_i(t)}

\newcommand{\tqyomi}{\tilde{q}_i(t,y_1(t)-)}

\newcommand{\wi}{W_i}
\newcommand{\wapi}{W_i(t,\bar{x}_{\alpha}(t)+)}
\newcommand{\wami}{W_i(t,\bar{x}_{\alpha}(t)-)}
\newcommand{\wbpi}{W_i(t,0+)}

\newcommand{\wyomi}{W_i(t,y_1(t)-)}

\newcommand{\lami}{\lambda_i(t,\bar{x}_{\alpha}(t)-)}
\newcommand{\lapi}{\lambda_i(t,\bar{x}_{\alpha}(t)+)}
\newcommand{\lbpi}{\lambda_i(t,0+)}

\newcommand{\lyomi}{\lambda_i(t,y_1(t)-)}

\DeclareMathOperator*{\tvv}{\mathrm{Tot.Var.}}

\DeclareMathOperator*{\cl}{cl}
\def\rn#1{\mathbb{R}^{#1}}
\newcommand{\tv}[2]{\ensuremath{\tvv_{#2}{\left(#1\right)}}}

\newcommand{\sumin}{\sum^n_{i=1}}
\newcommand{\phu}[1]{\ensuremath{\Phi_h(#1)}}
\newcommand{\ab}[1]{\ensuremath{\left|#1\right|}}
\newcommand{\lrb}[1]{\ensuremath{\left(#1\right)}}

\newcommand{\ift}{\text{if}}

\newcommand*{\pt}{\ensuremath{\partial_t}}      
\newcommand*{\px}{\ensuremath{\partial_x}}  

\begin{document}

\title{Local exact one-sided boundary null controllability of entropy solutions to a class of hyperbolic systems of balance laws
}

\titlerunning{Local exact one-sided boundary null controllability of entropy solutions}        

\author{Tatsien Li         \and
        Lei Yu 
}


\institute{Tatsien Li \at
              School of Mathematical Sciences, Fudan University, Shanghai 200433, China \\
              \email{dqli@fudan.edu.cn}            \\
            \emph{Supported by the National Basic Research Program of China (No 2013CB834100) and the National Natural Science Foundation of China (No. 11121101).} 
           \and
           Lei Yu (corresponding author)\at
             School of Mathematical Sciences, Tongji University, Shanghai 200092, China\\
             \email{yu\_lei@tongji.edu.cn}\\
              \emph{Supported by
              	the National Natural Science Foundation of China (No. 11501122).}
}

\date{Received: date / Accepted: date}

\maketitle
\hyphenpenalty=5000
\tolerance=1000
\maketitle

\begin{abstract}
We consider $ n\times n $ hyperbolic systems of balance laws in one-space dimension
\begin{equation}
\label{e:abst}
\pt H(u) +\px F(u)=G(u),\qquad t> 0,\ 0<x<L
\end{equation} 
under the assumption that all negative (resp. positive) characteristics are linearly degenerate. We prove the local exact one-sided boundary null controllability of entropy solutions to this class of systems, which generalizes the corresponding results obtained in \cite{Li-Yu_OC} from the case without source terms to that with source terms. In order to apply the strategy used in \cite{Li-Yu_OC}, we essentially modify the constructive method by introducing two different kinds of approximate solutions to system \eqref{e:abst} in the forward sense and to the system
\[
\px F(u)+\pt H(u)=G(u),\qquad 0<x<L,\ 0<t<T
\]
in the rightward (resp. leftward) sense, respectively, and we prove that their limit solutions are equivalent to some extent.
\end{abstract}

\section{Introduction}
In this paper, we study the local exact one-sided boundary null controllability of entropy solutions to $n\times n$ quasilinear hyperbolic system of balance laws in one space dimension: 
\begin{equation}
\label{e:intro}
\pt H(u)+\px F(u)=G(u),\qquad t> 0,\ 0<x<L,
\end{equation}
where $u$ is an $n$-vector valued unknown function of $ (t,x) $, $H,F$ and $G $ are smooth $n$-vector valued functions of $ u $, defined on a ball $B_r(0)$ centered at the origin in $\rn n$ with suitable small radius $r>0$. Moreover, we assume that 
\[
G\in \CC^2, \quad G(0)=0,
\]	
which means that $ u=0 $ is an equilibrium state.
For system \eqref{e:intro}, we give the following assumptions: 

$\textbf{(H1)}$ System \eqref{e:intro} is strictly hyperbolic, that is, for any given $u\in B_r(0)$, the matrix $DH(u)$ is non-singular and the matrix $(DH(u))^{-1}DF(u)$ has $n$ distinct real eigenvalues $  \la_i(u)$ $ (i=1,...,n) $: 
\begin{equation*}
\label{h:ev}
\la_1(u)<\cdots<\la_n(u),  \qquad \forall u\in \brz. 
\end{equation*}

$\textbf{(H2)}$ There are no zero eigenvalues, without loss of generality, in this paper we assume that there exist an $m\in \{1,...,n-1\}$ and a constant $c>0$, such that
\begin{equation}
\label{h:nonzero-char}
\lambda_m(u)<-c<0<c<\lambda_{m+1}(u), \qquad \forall u \in \brz.
\end{equation}
Under this assumption $ DF(u) $ is also a non-singular matrix.

$\textbf{(H3)}$ All characteristics are either genuinely nonlinear or linear degenerate in the sense of Lax \cite{Lax1987}. Recall that the $\ith$ characteristic is \emph{genuinely nonlinear} if
\begin{equation*}
\label{h:GN}
D \lambda_i(u)\cdot r_i(u) \ne 0,\qquad \forall u\in \brz,
\end{equation*}
while, the $\ith$ characteristic is \emph{linearly degenerate} if
\begin{equation*}
\label{h:LD}
D \lambda_i(u)\cdot r_i(u) \equiv 0,\qquad \forall u\in \brz,
\end{equation*}
where $ r_i(u) $ denotes the eigenvector of $ (DH(u))^{-1}DF(u) $, corresponding to $ \lambda_i(u)\ (1\le i\le n) $.

$\textbf{(H4)}$ Assume that system \eqref{e:intro} possesses an \emph{\textbf{entropy-entropy flux}}. Recall that  $ (\eta(u),\zeta(u)) $ is an entropy-entropy flux of system \eqref{e:intro} if  $\eta:\brz\to \R$ is a continuously differentiable convex function and  $\zeta:\brz\to \R$ is a continuously differentiable function, which satisfies
\begin{equation*}
\label{d:nc-entropy}
\begin{split}
D\eta(u)(DH(u))^{-1} DF(u) = D  \zeta(u), \quad \forall u\in \brz.
\end{split}
\end{equation*}

\vspace{12pt}

By \eqref{h:nonzero-char}, the boundaries $ x=0 $ and $ x=L $ are non-characteristic. We prescribe the following general nonlinear boundary conditions:
\begin{align}\label{bc:intro}
\begin{cases}
x=0: & b_1(u)=\go(t),\\ 
x=L: &  b_2(u)=\gt(t),
\end{cases}
\end{align}
where $g_1:\R^+\to \rn{n-m}$, $g_2:\R^+\to \rn{m}$ are given functions of $ t $, and $b_1 \in \C1(B_r(0)$; $\rn{n-m}),\ b_2 \in \C1(B_r(0); \rn{m})$. Here the value of $u(t,0)$ and $ u(t,L)$ should be understood as the inner trace of the function $u(t,x)$ on the boundaries $x=0$ and $ x=L$, respectively. In order to guarantee the well-posedness for the forward mixed initial-boundary value problem of system \eqref{e:intro}, we assume that 

$\textbf{(H5)}$ $ b_1 $ and $ b_2 $ satisfy the following conditions, respectively (see \cite{Li-Yu_boundary-value}):
\begin{equation}
\label{h:bc}
\begin{split}
& \det \left[
Db_1(u) \cdot r_{m+1}(u),\cdots,Db_1(u) \cdot r_{n}(u)
\right] \neq 0, \\
& \det \left[
Db_2(u) \cdot r_{1}(u), \cdots,Db_2(u) \cdot r_{m}(u)
\right] \neq 0,
\end{split}
\quad \forall u\in \brz.
\end{equation}

Now, we give the definition of entropy solution to the mixed initial-boundary value problem of system \eqref{e:intro}:
\begin{definition}
	\label{d:es}	
	For any given $T>0$, a function $u=u(t,x)\in \lnorm{1}((0,T)\times (0,L))$ is an \emph{entropy solution} to system \eqref{e:intro} on the domain $\D_T:=\{ (t,x)\ |\ 0< t< T,\ 0<x<L\}$ if
	\begin{enumerate}[(1)]
		\item the function $u$ is a weak solution to \eqref{e:intro} in the sense of distributions on the domain $ \D_T $, that is, for every $\phi \in C_c^1(\D_T)$ we have
		\[
		\int^T_0\int^L_0\pt\phi(t,x)H(u(t,x))+\px\phi(t,x)F(u(t,x))+\phi(t,x)G(u(t,x))dxdt=0,
		\]
		where $C^1_c(\D_T)$ denotes the set of $C^1$-functions with compact support in $\D_T$;
		\item the function $u$ is \emph{entropy-admissible} in the sense that there exists an entropy-entropy flux $ (\eta(u),\zeta(u)) $ for system \eqref{e:intro}, such that for every non-negative function $\phi \in C_c^1(\D_T)$, we have
		\begin{equation*}
		\begin{split}
		\label{d:entropy}
		\int^T_0\int^L_0\pt\phi(t,x) \eta(u(t,x))&+\px\phi(t,x) \zeta(u(t,x))dxdt\\
		&+\phi(t,x) D\eta(u(t,x))G(u(t,x))\geq 0.
		\end{split}
		\end{equation*}
	\end{enumerate}
	
	Moreover, if $ u $ also satisfies the initial condition
	$ \lim\limits_{t\to0+}u(t,x)=\iu(x)$ for a.e. $x\in (0,L)$ and the boundary conditions
	\begin{equation*}
	\lim_{x\to 0+}b_1(u(t,x))=\got,\qquad \lim_{x\to L-}b_2(u(t,x))=\gtt,\qquad \text{a.e. $t \in (0,T)$},
	\end{equation*}
	we say that $ u $ is \emph{an entropy solution to the mixed initial-boundary value problem} of system \eqref{e:intro} on the domain $\D_T$, associated with the initial condition
	\begin{equation}\label{ic:intro}
	t=0:\ u=\iu,\quad 0<x<L
	\end{equation}
	and the boundary conditions \eqref{bc:intro} for $ t\in (0,T) $. 
\end{definition}


The study of exact boundary controllability of entropy solutions to quasilinear hyperbolic systems of conservation laws was initiated by Bressan and Coclite and in \cite{Bressan_boundary-control-CL} they proved that for a class of quasilinear hyperbolic system of conservation laws (including the Diperna's system), there exists a class of initial data with infinite many jumps and with small total variation, such that it is impossible to drive the entropy solution to any constant state in finite time by boundary controls as far as the total variation of the entropy solution remains small. This implies that one can not expect the exact boundary controllability of entropy solutions for general quasilinear hyperbolic systems of conservation laws as in the case of classical solutions (see \cite{Li_controllability-book,Li_cam2002,Li_exact-controllability-quasilinear,Li_strong-weak}). However, some related results were obtained for special models of hyperbolic conservation laws, for example, scalar equation \cite{AnconaMarson1998,glass-burgers2007,Horsin,Leautaud-vanishing2012}, Temple system \cite{AnconaCoclite2005} and Euler equations \cite{Glass2007,Glass2014}. 

Recently, in \cite{Li-Yu_OC,Yu_CM2016}, we proved the local exact one-sided boundary null controllability of entropy solutions to
a class of general quasilinear hyperbolic systems of conservation laws ($G\equiv 0$ in \eqref{e:intro}) satisfying (H1)-(H5) under the assumption that all negative (resp. positive) characteristic families are linearly degenerate. 
In this paper, we will generalize the corresponding result to the case when there is a source term in system \eqref{e:intro}. It is a hard task, and to the authors' knowledges, so far there is no result concerning the exact boundary controllability of entropy solutions to quasilinear hyperbolic systems of balance laws. The main result of this paper is the following theorem.

\begin{theorem}
	\label{t:cl}
	Under assumptions (H1)-(H5) and the assumption that all negative characteristics are linear degenerate, if $ \gamma= \|G(u)\|_{W^{1,\infty}(\brz)} $ is sufficiently small and
	\begin{equation}
	\label{d:time-cl}
	T>L \left\{\frac{1}{|\lambda_m(0)|}+\frac{1}{\lambda_{m+1}(0)}\right\},
	\end{equation}
	then for any given initial data $\iu \in BV(0,L)$ with $\displaystyle\tv{\iu}{0<x<L}+|\iu(0+)|$ sufficiently small, 
	there exists a boundary control $g_2\in BV(0,T)$ with $\displaystyle\tv{g_2}{0<t<T}+|b_2(\iu(L-))-g_2(0+)|$ sufficiently small, acting on the boundary $x=L$, such that system \eqref{e:intro} together with the initial condition \eqref{ic:intro}
	and the boundary conditions
	\begin{equation*}
	\begin{cases}
	x=0:\ b_1(u)=0,\\
	x=L:\ b_2(u)=g_2(t),
	\end{cases}
	\qquad t\in (0,T)
	\end{equation*}
	admits an entropy solution $ u=u(t,x) $ on the domain $\D_{T}=\{\ 0<t<T,0<x<L\}$, which satisfies the final condition:
	\begin{equation}
	\label{e:fc}
	t=T:\quad u\equiv 0, \qquad \forall x\in (0,L).
	\end{equation}\
\end{theorem}

Similarly, under the assumption that all positive characteristic are linearly degenerate, we can obtain the local exact one-sided boundary null controllability of entropy solutions by means of a boundary control $g_1$ acting on the boundary $x=0$, that is, for any given initial data $\iu \in BV(0,L)$ with $\displaystyle\tv{\iu}{0<x<L}+|\iu(0+)|$ sufficiently small, this class of system \eqref{e:intro} admits an entropy solution satisfying the initial condition \eqref{ic:intro}, the final condition \eqref{e:fc} and the boundary conditions
\[
\begin{cases}
x=0:\ b_1(u)=g_1(t),\\
x=L:\ b_2(u)=0,
\end{cases}
\quad t\in (0,T).
\]

Roughly speaking, we will prove the above theorem by applying a modified version of the constructive method, originally introduced by Li and Rao in \cite{Li_controllability-book,Li_cam2002,Li_exact-controllability-quasilinear,Li_strong-weak} to study the exact boundary controllability of classical solutions to general quasilinear hyperbolic systems. This method is based on  the following three basic ingredients:
\begin{enumerate}[(1)]
	\item The well-posedness of semi-global solutions to the mixed initial-boundary value problem for quasilinear hyperbolic systems, where the semi-global solution means a solution existing in a time interval $ [0,T] $ with some preassigned time $ T>0 $. 
	\item The fact that the semi-global solution $ u=u(t,x) $ to system \eqref{e:intro} in the forward sense on the domain $ \D_T $ is also a semi-global solution to system
	\begin{equation}
	\label{e:right-intro}
	\px F(u)+\pt H(u)=G(u)
	\end{equation}
	in the rightward sense 
	on the same domain, where the roles of $ t $ and $ x $ are exchanged, namely, $ x $ is regarded as the ``time" variable and $ t $ is regarded as the ``space" variable. And vice versa.
	\item  The determinate domain of semi-global solutions to the one-sided rightward mixed initial-boundary value problem of system  \eqref{e:right-intro}.
\end{enumerate}

Here the semi-global solution in the forward sense and that in the rightward sense are obtained as the limits of two kinds of approximations, respectively. The reason that we use different approaches to construct semi-global solutions will be explained in the sequel. 

In\cite{Li-Yu_OC}, we obtained the local exact one-sided boundary null controllability of entropy solutions to a class of quasilinear hyperbolic systems of conservation laws 
\begin{equation}\label{e:hcl}
\pt H(u)+\px F(u)=0,\qquad t\ge 0,\ 0<x<L
\end{equation}
which corresponds to $ G\equiv 0 $ in \eqref{e:intro}, by proving the previous three ingredients in the case without source terms. In fact, we proved the semi-global well-posedness of entropy solution as the limit of $ \e $-approximate front tracking solutions. Although the ingredient (2) is trivial in the case of classical solutions, it is not always true for entropy solutions which are irreversible with respect to the time in general. Moreover, the rightward system \eqref{e:right-intro} may not admit an entropy-entropy flux even if the forward system \eqref{e:intro} does possess an entropy-entropy flux, no matter there is a source term $ G $ or not. To avoid checking the entropy-admissible condition directly, we considered the solution as the limit of $ \e $-approximate front tracking solutions and proved that the $ \e $-approximate front tracking solution to system \eqref{e:hcl} in the forward sense is also an $ \e $-approximate front tracking solution to system \eqref{e:right-intro} with $ G\equiv 0 $ in the rightward sense under the assumption that all negative characteristics are linearly degenerate. Then by passing to the limit, we obtained the ingradient (2) for system \eqref{e:hcl}. This argument makes sense since a solution of system \eqref{e:hcl} as the limit of $ \e $-approximate solutions is actually an entropy solution when system \eqref{e:hcl} process an entropy-entropy flux.

In this paper, combining the semi-global existence of entropy solutions proved in \cite{Li-Yu_OC} and the splitting method, we will generalize the corresponding semi-global existence of entropy solution to the system of balance laws \eqref{e:intro} by considering the solution as the limit of $ \he $-solutions which are, in some extend, piecewise constant approximate solutions to system
\begin{equation*}\label{e:he-apx}
\pt H(u)+\px F(u)=\sum_{j\in\Z\cap [0,T/h]}h G(u)\delta_0(t-jh),
\end{equation*}
where $ \delta_0$ denotes the delta function. Concerning the ingredient (2), the argument used in the case of conservation laws does not work for balance laws, since an $ \he $-solution to system \eqref{e:intro} in the forward sense is no longer the $ \he $-solution to system \eqref{e:right-intro} in the rightward sense. So for the rightward system \eqref{e:right-intro}, we consider the solution as the limit of $ \eh $-solutions which are piecewise constant approximate solutions to system
\begin{equation*}\label{e:eh-apx}
\px F(u)+\pt H(u)=\sum_{j\in\Z\cap [0,T/h]}h G(u)\delta_0(t-jh) ,
\end{equation*}
and we prove that an $ \he $-solution to system \eqref{e:intro} in the forward sense is also the $ \eh $-solution to system \eqref{e:right-intro} in the rightward sense on domain $ \D_T $, provided that all negative characteristics are linearly degenerate, and vise versa. By passing to the limit, we obtain the desired ingredient (2). Moreover, we prove the ingredient (3) only for the solutions as the limit of $ \eh $-solutions to the rightward system \eqref{e:right-intro}. 

The paper is organized as follows. In Section \ref{s:pre}, we recall the construction of $ \e $-solutions to the mixed initial-boundary value problem of quasilinear hyperbolic systems of conservation laws \eqref{e:hcl}. 
Section \ref{s:inhom} concerns the well-posedness of entropy solution to the system of balance laws \eqref{e:intro} and it contains three parts. In Section \ref{ss:hes}, we describe the construction of $ \he $-solutions to the quasilinear hyperbolic system of balance laws \eqref{e:intro}. Then by passing to the limit, we obtain the existence of the semi-global entropy solution when system \eqref{e:intro} possesses an entropy-entropy flux. In Section \ref{ss:ehs} we introduce another kind of piecewise constant approximate solutions, called $ \eh $-solutions, and prove the semi-global existence and stability of solution as the limit of $ \eh $-solutions. In Section \ref{ss:rbhesehs}, we prove the equivalence between the solution as the limit of $ \he $-solutions to system \eqref{e:intro} in the forward sense and the solution as the limit of $ \eh $-solutions to system \eqref{e:right-intro} in the rightward sense.
In Section \ref{s:oneside}, using the theory of semi-global solutions obtained in Section \ref{s:inhom}, we prove the local exact one-sided boundary null controllability under the assumption that all negative characteristic are linearly degenerate. In Appendix, we collect all the technical proofs for lemmas and propositions stated in Section \ref{s:inhom}.  

Throughout this paper, in order to avoid abusively using constants, we denote by the notation $C$ a positive constant which depends only on system \eqref{e:intro}, constant $ L $ and functions $b_1,b_2$, but is independent of the special choice of initial data $ \iu $, boundary data $
\go,\gt $ and time $T$. Moreover, we denote by $ C(t) $ a positive constant which depends also on time $ t $.

\section{Preliminaries}
\label{s:pre}
\subsection{Basic notations}

In this section, under Assumptions (H1)-(H5), we consider the semi-global solution to the following mixed initial-boundary value problem  of system
\begin{equation}\label{e:gs}
\pt H(u)+\px F(u)=G(u),\quad t> 0,\ 0<x<L
\end{equation}
with the initial-boundary conditions
\begin{equation}
\label{e:ibc}
\begin{cases}
t=0:\ u=\iu(x),& 0<x<L,\\
x=0:\ b_1(u)=\go(t),& t>0,\\
x=L:\ b_2(u)=\gt(t),& t>0.
\end{cases}
\end{equation}

We normalize the left and right eigenvectors $l_i(u)$ and $ r_i(u)$ $ (i=1,...,n) $ of $(DH)^{-1}DG(u)$, so that
\[
l_i(u) \cdot r_j(u) \equiv\delta_{ij},\quad i,j =1,...,n,
\]
where $ \delta_{ij} $ is the Kronecker symbol.
For two states $\omega,\omega'\in \rn n$, let
\begin{equation*}
\label{def:average-matrix-DH}
A(\omega,\omega')=\int^1_0 [DH](\theta \omega+(1-\theta)\omega') d\theta
\end{equation*}
and
\begin{equation*}
\label{def:average-matrix-DG}
B(\omega,\omega')=\int^1_0 [DG](\theta \omega+(1-\theta)\omega') d\theta.
\end{equation*}
We have
\[
H(\omega')-H(\omega)=A(\omega,\omega')(\omega'-\omega),\quad
G(\omega')-G(\omega)=B(\omega,\omega')(\omega'-\omega).
\]
By hyperbolicity, for two $ \omega $ and $ \omega' $ sufficiently close to the origin, we denote by $\lambda_i(\omega,\omega')$ the $i$-th real eigenvalue of the matrix $A^{-1}(\omega,\omega')B(\omega,\omega')$ (cf. \cite{Li-Yu_boundary-value}).

For any given $u\in \brz$, let $\sigma\mapsto R_i(\sigma)[u]$ denote the $i$-rarefaction curve passing through $u$ for $ \sigma\in[-\sigma_0,\sigma_0] $ with $ \sigma_0 $ sufficiently small and let $\sigma\mapsto S_i(\sigma)[u]$ denote the $i$-shock curve passing through $u$ for $ \sigma\in[-\sigma_0,\sigma_0] $ with $ \sigma_0 $ sufficiently small (cf. \cite{Li-Yu_OC}).
If the $\ith$ characteristic is linearly degenerate, we choose $r_i(u)$ to have the unit length and we let the coinciding $i$-rarefaction curve and $i$-shock curve be parameterized by the arc-length. If the $\ith$ characteristic is genuinely nonlinear, we choose $r_i(u)$ such that $\nabla\lambda_i(u)\cdot r_i(u)\equiv 1$ and we let the $i$-rarefaction curve and the $i$-shock curve be parameterized in such a way that
\begin{eqnarray*}
\lambda_i \big( R_i(\sigma)[u]\big) - \lambda_i(u)=\sigma\quad \text{and}\quad \lambda_i\big( S_i(\sigma)[u]\big)-\lambda_i(u)=\sigma,
\end{eqnarray*}
respectively. This parametrization leads to a useful property:
\begin{equation*}
\label{property-parameter}
u=S_i(-\sigma)S_i(\sigma)[u]\qquad \text{for all }\sigma\in [-\sigma_0,\sigma_0],\ u\in \brz,\ i\in \{1,...,n\}.
\end{equation*}
With this parametrization, a straightforward computation shows that the composite function
\begin{equation*}
\Psi_i(\sigma)[u]=
\begin{cases}
R_i(\sigma)[u] & \ift\ \sigma>0,\\
S_i(\sigma)[u] & \ift\ \sigma<0
\end{cases}
\end{equation*}
is smooth for $\sigma\ne 0$ and of class $C^2$ at $\sigma=0$, which is called the \emph{$\ith$ elementary wave curve}.

\subsection{$ \e $-solutions for homogeneous systems}
\label{ss:es}
In \cite{Li-Yu_OC}, under the assumptions (H1)-(H5), we obtained the well-posedness of semi-global entropy solutions to the mixed initial-boundary value problem of general quasilinear hyperbolic systems of conservation laws \eqref{e:hcl}
associated with the initial-boundary conditions \eqref{e:ibc}. This was done by the approximation of $ \e $-solutions (that is, $ \e $-approximate front tracking solutions therein). Recall the definition of $ \e $-solutions to system \eqref{e:hcl}.  

\begin{definition}
	\label{d:es-h}
	For any given time $ T>0 $ and any fixed $\e>0$, we say that a continuous map
	\[
	t\mapsto \ue(t,\cdot)\in \lnorm 1(0,L), \quad \forall t\in (0,T)
	\]
	is an \emph{$\e$-solution} to the quasilinear hyperbolic system of conservation laws \eqref{e:hcl} on the domain $ \D_T $, if 
	\begin{enumerate}[(1)]
		\item as a function of two variables, $\ue=\ue(t,x)$ is piecewise constant with discontinuities occurring along finitely many straight lines with non-zero slope in the domain $ \D_T$ and only finitely many interactions of fronts occur. Jumps can be of three types: shocks (or contact discontinuities), rarefaction fronts, non-physical fronts which are denoted by $\ms$, $ \mr $ and $\NP $, respectively. $ \P=\ms\cup \mr $ are called physical fronts.
		\item along each physical front $x=\xat\ (\alpha \in \P)$, the left and right limits $\ul:=\ue(t,\xat-)$ and $\ur:=\ue(t,\xat+)$ of $\ue(t,\cdot)$ on it are selected by
		\begin{equation*}
		\ur=\Psi_{k_{\alpha}}[\sigma_{\alpha}]\big(\ul\big),
		\end{equation*}
		where  $k_{\alpha}\in\{1,...,n\}$ denotes the corresponding family of the front, and $\sigma_{\alpha}$ is its \emph{wave amplitude}. Moreover, if the $k_{\alpha}$-th characteristic family is genuinely nonlinear with $\sa<0$ (i.e., $ \alpha\in \ms$) or linearly degenerate, then
		\begin{equation*}
		\label{d:as}
		|\dot x_\alpha(t)-\lambda_{k_{\alpha}}(\ul,\ur)(t)|\le \oo \e,\quad \forall t\in (0,T);
		\end{equation*}
		while, if the $k_{\alpha}$-th characteristic family is genuinely nonlinear with $0<\sa\le \e$ (i.e., $ \alpha\in \mr $), then the speed of the front satisfies
		\begin{equation*}
		\label{d:ar}
		|\dot x_\alpha(t)-\lambda_{k_{\alpha}}(\ur)(t)|\le \oo \e, \quad \forall t\in (0,T), 
		\end{equation*}
		where $ C $ stands for a positive constant.
		\item all non-physical fronts $x=\xat\ (\alpha\in \NP)$ have the constant speed $\dot x_{\alpha}\equiv \hat \lambda$ with $0<\hat \lambda<c$, where $ c $ is given by \eqref{h:nonzero-char}. Moreover, the total amplitude of all non-physical waves in $\ue(t,\cdot)$ is uniformly bounded with respective to $ t\in (0,T) $ by $\e$, i.e.,
		\[
		\sum_{\alpha\in\NP}|\ue(t,\xa+)-\ue(t,\xa-)|\le \e,\qquad \forall t\in (0,T).
		\]
	\end{enumerate}
\end{definition}

For any given small $ \e>0 $ and any given initial-boundary data $ (\iu,\go,\gt) $ with $ \lug $ sufficiently small,
where
\begin{equation}
\label{def:lug}
\begin{split}
\lug:=\tv{\iu(\cdot)}{0<x<L}+&|\iu(0+)|+\sum_{i=1,2}\tv{g_i(\cdot)}{0<t<T}\\
+&|b_1(\iu(0+))-\go(0+)|+|b_2(\iu(L-))-\gt(0+)|,
\end{split}
\end{equation}
we proved in \cite{Li-Yu_OC} the semi-global existence of $ \e $-solutions to the initial-boundary value problem of quasilinear hyperbolic system of conservation laws \eqref{e:hcl} and \eqref{e:ibc} on the domain $ \D_T $. In fact, we can construct an $\e$-solution by the following algorithm: for any given positive time $ T $, we choose the approximate initial-boundary data $ (\iue,\goe,\gte) $ with finitely many jumps, such that
\begin{subequations}\label{h:ib-app}
	\begin{align}
	&\tv{\iue(\cdot)}{0<x<L} \leq \tv{\iu(\cdot)}{0<x<L}, \quad \|\iue(\cdot)-\iu(\cdot)\|_{L^{\infty}(0,L)} \le \e,\label{h:i-app}\\
	&\tv{\gie(\cdot)}{0<t<T} \leq \tv{g_i(\cdot)}{0<t<T}, \quad \|\gie(\cdot)-g_i(\cdot)\|_{L^{\infty}(0,T)}\le \e/2,\qquad i=1,2.\label{h:b-app}
	\end{align}
\end{subequations}
Let $x_1<\dots<x_p$ be the jump points of $\iue$ and write $ x_0=0$ and
$x_{p+1}=L $.  In order to obtain a piecewise constant approximate solutions locally, for $\alpha=0,1,\dots,p+1$, we solve the Riemann initial value problem at each point $(0,\xa)\ (\alpha=1,...,p)$ and the left-sided (resp. right-sided) Riemann mixed problem at the point $(0,x_{0})$ (resp. $ (0,x_{p+1})$), by means of the \emph{approximate Riemann solver} which is defined by replacing rarefaction waves in the exact solution with rarefaction fans (see also \cite{Bressan2000}).
The straight lines on which the discontinuities located are called \emph{fronts}. A front travels with constant speed until it meets another front or hits the boundary at the so-called \emph{inner/boundary interaction point}. Then the corresponding new Riemann problem will be approximately solved by a piecewise constant self-similar solution in an analogous manner. This procedure can be continued up to $t = T$ if the choice of approximate Riemann solver ensures the smallness of total variation of the approximate solution and only produces finitely many interactions. For this aim, besides the approximate Riemann solver, two additional Riemann solvers, \emph{simplified approximate Riemann solver} and \emph{crude Riemann solver}, are applied according to different situations. More precisely, both of them do not increase the number of physical fronts, and we apply the simplified approximate Riemann solver when the inner (resp. boundary) interaction involves only physical fronts, and the product of the strength of two incoming fronts (resp. the strength of the front hitting the boundary) is small, while, we apply the crude Riemann solver if the inner/boundary interaction involves also a non-physical front. Here, by slightly modifying the speed of fronts, we may assume that at any time before the time $T$, at most one of the following situations happens: (a) two fronts interact with each other inside the domain $\D_{T}$; (b) one front hit the boundary; (c) at least one of the boundary data $g^{\e}_{i}\ (i=1,2)$ has a jump. In particular, two non-physical fronts can not interact since they travel with the same constant speed. 

For notational convenience, in what follows, for a front (or wave) $\alpha$, the symbol $\sigma_{\alpha}$ will denote its amplitude, furthermore, we denote by $\mathcal{F}(\alpha)$ the characteristic family of the front (or wave) $\alpha$ and we say that a front $ \alpha $ is an $ i $-front if $ \mathcal{F}(\alpha)=i $.

Moreover, at the inner/boundary interact point, the $\e$-solution constructed above satisfies the following properties 
(see \cite{Colombo_general-balance-boundary} or \cite{Li-Yu_OC}):
\begin{enumerate}[(1)]
	\item\label{item:1} Suppose that an $ i $-fronts $\alpha$ and a $j$-front $ \beta $ interact with each other at an inner point in $\D_T$. Let $\alpha_1',...,\alpha_n'$ be the outgoing waves with $\mathcal{F}(\alpha_{\ell}')=\ell\ (\ell=1,...,n)$, generated by the approximate Riemann solver. Then we have 
	\begin{equation}
	\label{eq:2-1}
	| \sigma_{\alpha'_i}-\sigma_{\alpha} |+| \sigma_{\alpha'_j}-\sigma_{\beta} |+\sum\limits_{k\ne i,j}| \sigma_{\alpha'_k}|\le C | \sigma_{\alpha}\sigma_{\beta} |,\quad \text{if } i\ne j
	\end{equation}
	and
	\begin{equation}
	\label{eq:2-2}
	| \sigma_{\alpha'_i}-\sigma_{\alpha}-\sigma_{\beta}| +\sum\limits_{k\ne i}| \sigma_{\alpha'_k}|\le C | \sigma_{\alpha}\sigma_{\beta} |,\quad \text{if } i= j.
	\end{equation}
	\item\label{item:2} A front $\alpha$ (resp. $\beta$)  hits the right boundary $x=L$ (resp. the left boundary $x=0$) with $\mathcal{F}(\alpha)\ge m+1$ (resp. $\mathcal{F}(\beta)\le m$), where $ m $ is given by (H2). Let $\alpha_1',...,\alpha_m'$ (resp. $\alpha_{m+1}',...,\alpha_n'$) be the outgoing waves with $\mathcal{F}(\alpha_k')=k \ (k=1,...,m)$ (resp. $k=m+1,...,n$), generated by the approximate Riemann solver. Then we have
	\begin{equation}
	\label{eq:4}
	\sum\limits_{k=1}^m|\sigma_{\alpha'_k}|\le C |\sigma_{\alpha}| \qquad \text{(resp. $\sum\limits_{k=m+1}^n |\sigma_{\alpha'_{k}}|\le C |\sigma_{\beta}|$)}.
	\end{equation}
	\item\label{item:3} Suppose that a physical front $ \alpha $ interacts from left with a physical front $ \beta $. If the simplified approximate Riemann solver is used at the interaction point and generates a non-physical front $ \gamma $ and two corresponding outgoing physical waves ${\alpha'}$ and ${\beta'} $ when $ \mathcal{F}(\alpha)\ne \mathcal{F}(\beta) $ or a single physical wave ${\alpha'} $ when $ \mathcal{F}(\alpha)= \mathcal{F}(\beta) $, then we have
	\begin{equation*}
	|\sigma_{\alpha'} -\sa|+|\sigma_{\beta'}-\sb|+|\sigma_{\gamma}|\leq \oo |\sa\sb|\quad \text{if}\quad \mathcal{F}(\alpha)\ne \mathcal{F}(\beta)
	\end{equation*}
	and
	\begin{equation*}
	|\sigma_{\alpha'} -\sa-\sb|+|\sigma_{\gamma}|\leq \oo |\sa\sb|\quad \text{if}\quad \mathcal{F}(\alpha)= \mathcal{F}(\beta).
	\end{equation*}
	\item\label{item:4} Suppose that a non-physical front $ \gamma $ interacts with a physical front $ \alpha $. If $ \gamma' $ is the outgoing non-physical front defined by the crude Riemann solver, then we have
	\begin{equation*}
	\label{eq:6}
	|\sigma_{\gamma'}-\sigma_{\gamma}|\leq \oo |\sa\sigma_{\gamma}|.
	\end{equation*}
	\item\label{item:8} Suppose that at time $t=\tau$, $g^{\e}_1(\tau+)\ne g^{\e}_1(\tau-)$ (resp. $g^{\e}_2(\tau+)\ne g^{\e}_2(\tau-)$) and $\alpha'_{m+1},...,\alpha'_n$ (resp. $\alpha'_1,...,\alpha'_m$) are the outgoing fronts generated by the approximate Riemann solver. Then we have
	\begin{equation}
	\label{eq:1}
	\begin{split}
	&\qquad \sum\limits_{i=m+1}^n|\alpha'_i|\le C|g_1^{\e}(\tau+)-g_{1}^{\e}(\tau-)|\\
	&\left(\text{resp.}\quad \ \sum\limits_{i=1}^m|\alpha'_i|\le C|g_2^{\e}(\tau+)-g_2^{\e}(\tau-)|\right).
	\end{split}
	\end{equation}
	\item\label{item:6} Suppose that a physical front hits the boundary and the simplified approximate Riemann solver is applied, or a non-physical front hits the right boundary $x=L$ and the crude Riemann solver is applied, then no outgoing front will be generated.
\end{enumerate}
Finally, we can construct an $ \e $-solution $ \ue=\ue(t,x) $ to problem \eqref{e:hcl} and \eqref{e:ibc}, such that
\begin{gather*}
\ue(0,x)=\iue(x),\quad \text{for a.e. }x\in (0,L),\nonumber\\
\|b_1(\ue(\cdot,0+))-\goe(\cdot)\|_{L^{\infty}(0,T)}\le \e/2,\quad \|b_2(\ue(\cdot,L-))-\gte(\cdot)\|_{L^{\infty}(0,T)}\le \e/2\nonumber
\end{gather*}   
with some additional stability properties. 

The following lemma, concerning the equivalence between the $\e$-solution to system \eqref{e:hcl} in the forward sense and that to system
\begin{equation}
\label{e:hclr}
\px F(u)+ \pt H(u)=0
\end{equation}
in the rightward sense under the assumption that all the negative characteristic are linearly degenerate. This result will be useful in what follows for describing the relation between $ \he $-solutions and $ \eh $-solutions.
\begin{lemma}
	\label{l:uecv}
	Under assumptions (H1)-(H5) and the assumption that all negative characteristics are linearly degenerate, suppose that $\ue=\ue(t,x)$ is an $\e$-solution to system \eqref{e:hcl} in the forward sense on the domain $\D_T$. Then, if we exchange the role of $ t $ and $ x $, namely, regard $ x $ as the ``time" variable and $ t $ as the ``space" variable, $ \ue $ is also an $\e$-solution to system \eqref{e:hcl} in the rightward sense, i.e., an $ \e $-solution to system \eqref{e:hclr} on the domain $ \D_T $.
\end{lemma}

\section{Semi-global solutions to inhomogeneous systems}
\label{s:inhom}
\subsection{$ \he $-solutions}
\label{ss:hes}

\begin{definition}\label{d:he-es}
	For any given time $ T>0 $ and any fixed $\e>0$, we say that a continuous map
	\[
	t\mapsto \uhe(t,\cdot)\in \lnorm 1(0,L), \quad \forall t\in (0,T)
	\]
	is an \emph{$\he$-approximate solution} to the quasilinear   hyperbolic system of balance laws \eqref{e:gs} on the domain $ \D_T $ if 
	\begin{itemize}
		\item As a function of two variables, $\uhe=\uhe(t,x)$ is piecewise constant with discontinuities occurring along finitely many straight lines which are called \emph{fronts}.
		\item Restricted on the domain $ Z_j:=\{jh<t<(j+1)h,\ 0<x<L\}\cap \D_T,\ j\in \Z\cap (0,T/h)$, $ \uhe $ is an $ \e $-solution to the quasilinear hyperbolic system of conservation laws \eqref{e:hcl}.
		\item  Along the segment $ \mmz_j =\{t=jh\}\times\{0<x<L\}$, the values $ \ul=\uhe(jh-,x) $ and $ \ur=\uhe(jh+,x) $ satisfy $ \ur=H^{-1}[H(\ul)+hG(\ul)] $ for all $ x\in (0,L) $ except at the points where fronts interact with the segment $ \mmz_j $.
	\end{itemize}	
	Moreover, if the initial and boundary values of $\uhe$ satisfy approximately the initial-boundary conditions \eqref{e:ibc}, that is, if
	\begin{equation*}
	\label{d:uhe-ai}
	\|\uhe(0,\cdot)-\bar u(\cdot)\|_{\lnorm 1(0,L)}\leq \e,
	\end{equation*}
	\begin{equation*}
	\label{d:uhe-ab}
	\|b_1\big(\uhe(\cdot,0+)\big)-g_1(\cdot)\|_{\lnorm{1}(0,T)}\leq \e,\quad \|b_2\big(\uhe(\cdot,L-)\big)-g_2(\cdot)\|_{\lnorm{1}(0,T)}\leq \e,
	\end{equation*}
	then $\uhe=\uhe(t,x)$ is called the $\he$-solution to the mixed initial-boundary value problem \eqref{e:gs}-\eqref{e:ibc}. 
\end{definition}

\subsubsection{Construction of $ \he$-solutions}
\label{sss:ches}
Based on the existence of $ \e $-solutions for the mixed initial-boundary value problem \eqref{e:hcl} and \eqref{e:ibc} and the standard splitting algorithm, we can construct $ \he$-approximate solutions to problem \eqref{e:gs}-\eqref{e:ibc} as follows.

We construct an $ \he $-solution $ u=\uhe(t,x) $ to the balance laws \eqref{e:gs} with approximate initial-boundary data $ (\bu^{\e},g^{\e}_1,g^{\e}_2) $ by induction: we first  choose a piecewise constant vector function  $ (\bu^{\e},g^{\e}_1,g^{\e}_2) $ which is a good approximation to the given initial-boundary data $ (\iu,g_1,g_2) $ such that \eqref{h:ib-app} holds. Then we define $ \uhe(t,x) $ for $ t\in [0,h) $ as an $ \e $-solution of the quasilinear hyperbolic system of conservation laws \eqref{e:hcl} with the initial-boundary data  $ (\bu^{\e},g^{\e}_1,g^{\e}_2) $. Then, at the time $ t=h $ we define 
\[ \uhe(h+,x):=H^{-1}[H(\uhe(h-,x))+hG(\uhe(h-,x))]. \] Now\ suppose $ \uhe(t,x) $ is defined for $ t\in [kh,(k+1)h) $ as an $ \e $-solution to the system of conservation laws \eqref{e:hcl} with approximate initial-boundary data $ (\uhe(kh-,\cdot),\geo,\get)) $, then at the time $ t= (k+1)h $, $ \uhe $ is defined by  \[\uhe((k+1)h,x)=H^{-1}[H(\uhe((k+1)h-,x))+hG(\uhe((k+1)h-,x))]. \]

This procedure can be continued up to the time $ t=T $, if the number of interaction is finite and the total amplitude of wave-strength remains sufficiently small. This can be proved by the argument used in \cite{Li-Yu_OC}. Roughly speaking, we first estimate the total variation of $ \he $-solutions in a small time interval $ [0,\hatt/2] $, where
\begin{equation*}
\label{d:hattauone}
\hat{\tau}=\frac{L}{2}\min_{u\in \brz}\{|\lambda_1(u)|^{-1},\lambda_n(u)^{-1}\}.
\end{equation*}
This can be done by estimating the total amplitude of fronts on two trapezoid domains
\begin{equation}\label{d:lt}
\mfl:=\left\{(t,x)\ |\ 0<t<\hat \tau,0<x< L(\hat \tau-t)/{(2\hatt)}\right\}
\end{equation}
and
\begin{equation}\label{d:rt}
\mfr:=\left\{(t,x)\ |\ 0<t<\hat\tau,\ Lt/(2\hatt)<x<L \right\},
\end{equation}
which involves only one-sided boundary data, respectively. Noticing that $ \D_{\hatt}\subseteq \mfl\cup\mfr $, we obtain the smallness of total amplitude of fronts in $ \D_{\hatt} $ if we do have the corresponding results on $ \mfl $ and $ \mfr $. Then we can estimate the total variation of $ \uhe $ by induction. The details of proof can be found in Appendix. In fact, we have the following
\begin{proposition}
	\label{p:uhe-ex}
	For any fixed $T>0$, there exist positive constants $\delta$ and $ C(T) $ such that for every initial-boundary data $ (\iu,\go,\gt) $ with
	\begin{equation*}
	\lug \le \delta
	\end{equation*}
	(cf. \eqref{def:lug}) and for any given $\e>0$ and $ h>0 $ sufficiently small, there exists an $\he$-solution  $\uhe=\uhe(t,x)$ to the mixed initial-boundary value problem \eqref{e:gs}-\eqref{e:ibc} on the domain $ \D_T $, satisfying
	\begin{subequations}
		\begin{align}
		&\tv{\uhe(t,\cdot)}{0<x<L} \le C(t) ( \lug+\gamma t), \quad \forall t\in (0,T),\label{bd:uhe-tvx}\\
		&\tv{\uhe(\cdot,x)}{0 <t<T}  \le C(T) (\lug+\gamma T), \quad \forall  x\in (0,L),\label{bd:uhe-tvt}\\
		&\|\uhe(t,\cdot)-\uhe(s,\cdot)\|_{\lnorm 1(0,L)} \leq C(T)  |t-s|+Ch,\quad  \forall t,s \in(0,T), \label{bd:uhe-L1conti-t}\\
		& \|\uhe(\cdot,x)-\uhe(\cdot,x')\|_{\lnorm 1(0,T)}  \leq C(T) |x-x'|, \quad \forall x,x' \in (0,L) \label{bd:uhe-L1conti-x}
		\end{align}
	\end{subequations}
	and $ \uhe \in \brz$ for a.e. $ (t,x)\in \D_T $.
\end{proposition}

\subsubsection{Solution as the limit of $\he$-solutions}
\label{sss:slhes}

Fix a sequence of pair $ (h^{\mu},\enu) $ with $ \enu\to 0 $ as $ \nu\to +\infty $ and $ h^{\mu}\to 0 $ as $ \mu\to +\infty $. For each $ \mu\ge 1 $ and $\nu\ge 1$, Proposition \ref{p:uhe-ex} yields an $(h^{\mu},\enu)$-solution $u^{\mu,\nu}$ to the mixed initial-boundary value problem \eqref{e:gs}-\eqref{e:ibc}, such that for all $ \mu,\nu\ge 1 $ the maps $t \mapsto u^{\mu,\nu}(t,\cdot)$ are uniformly Lipschitz continuous in $ \lnorm 1 $ norm with respect to $ t $, and $ \displaystyle{\tv{u^{\mu,\nu}(t,\cdot)}{0<x<L} }$ remains sufficiently small uniformly for all $ t\in (0,T) $. By Helly's Theorem \cite[Theorem 2.3]{Bressan2000}, for each fixed $ \mu\ge 1 $, we can extract a subsequence of $ \{u^{\mu,\nu}\} $, which converges to a limit function $u^{\mu}=u^{\mu}(t,x)$ in $\lnorm 1((0,T)\times(0,L))$ as $ \nu\to +\infty $ and the estimates \eqref{bd:uhe-tvx} and \eqref{bd:uhe-L1conti-t} still hold for the limit function $u^{\mu}$. Then using Helly's Theorem again, we can extract a convergent subsequence of $ \{u^{\mu}\} $. By a standard argument (see \cite{Amadori-unique_bl}), it can be proved that the limit function $u=u(t,x)$ is in fact an entropy solution to problem \eqref{e:gs}-\eqref{e:ibc}.

\begin{theorem}
	\label{t:luhe-ex}
	For any fixed $T>0$, there exist positive constants $\delta$ and $ C(T) $ such that for every initial-boundary data $(\iu,g_1,g_2)$ with
	\begin{equation*}
	\Lambda(\iu,g_1,g_2) \le \delta,
	\end{equation*}
	problem \eqref{e:gs}-\eqref{e:ibc} associated with $ (\iu,g_1,g_2) $ admits a solution $ u=u(t,x) $ on the domain $ \D_T$ as the limit of $\he$-solutions provided by Proposition \ref{p:uhe-ex}, satisfying
	\begin{align*}
	& \tv{u(t,\cdot)}{0<x<L} \le C(t)  \lug, \quad \forall t\in (0,T),
	\\
	& \tv{u(\cdot,x)}{0 <t<T} \le C(T) \lug+\gamma T, \quad \forall  x\in (0,L),
	\\
	& \|u(t,\cdot)-u(s,\cdot)\|_{\lnorm 1(0,L)}\leq C(T)  |t-s|,\quad \forall t,s \in (0,T), 
	\\
	& \|u(\cdot,x)-u(\cdot,x')\|_{\lnorm 1(0,T)}\leq C(T) |x-x'|,\quad \forall x,x' \in (0,L), 
	\end{align*}
	and $u(t,x)\in \brz$ for a.e. $(t,x)\in \D_T  $.
\end{theorem}

\subsection{$ \eh $-solutions}
\label{ss:ehs}

\subsubsection{Inhomogeneous Riemann solver}

Recall that the construction of $ \e $-solutions of  quasilinear hyperbolic system of conservation laws \eqref{e:hcl} is based on Riemann solvers which solve approximately the Riemann problem of system \eqref{e:hcl} associated with a piecewise constant initial data 
\begin{equation}\label{e:ri}
u(0,x)=
\begin{cases}
\ul & \text{if } x<x_0,\\
\ur & \text{if } x>x_0
\end{cases}
\end{equation}
for any given $ x_0\in(0,L) $.

For the quasilinear hyperbolic system of balance laws \eqref{e:gs}, following the idea used in \cite{Amadori-BV-BL}, we now take into account the effect of the source term by using $h$-Riemann solver which approximately solves the Riemann problem  \eqref{e:gs} and \eqref{e:ri} by introducing an extra stationary discontinuity, called \emph{zero wave}, across the line $ x=x_0 $.  More precisely, as the map $ u \mapsto F(u) $ is invertible for $ u\in \brz $, for small $ h>0 $ we can define
\begin{equation*}\label{key}
\phu{u}:=F^{-1}\left[ F(u)+G(u)h\right].  
\end{equation*}
It is easy to see that  for any fixed $ h>0 $, the $ C^2 $-norm of the map $ u\mapsto \phu{u} $ is bounded by a positive constant independent of $ h $, and the map $h\mapsto \phu{u} $ is Lipschitz continuous. We say that $ u=u(t,x) $ is an \emph{$ h $-Riemann solver} for problem  \eqref{e:gs} and \eqref{e:ri}, if the following conditions hold:
\begin{enumerate}[(a)]
	\item there exist two states $ u^- $ and $ u^+ $ satisfying $ u^+=\phu{u^-} $;
	\item $ u $ coincides with the solution to the homogeneous Riemann problem \eqref{e:hcl} with the initial data $ [\ul,u^-] $ on the set $ \{0<t<T, x<x_0\} $ and in the mean time, $ u $ coincides with the solution to the homogeneous Riemann problem \eqref{e:hcl} with initial data $ [u^+,\ur] $ on the set $ \{ 0<t<T, x>x_0\} $;
	\item the Riemann problem between $ \ul $ and $ u^- $ is solved only by waves with negative speed;
	\item the Riemann problem between $ u^+ $ and $ \ur $ is solved only by waves with positive speed.
	
\end{enumerate}

The next lemma concerns the existence and uniqueness of the $ h $-Riemann solver.
\begin{lemma}[\cite{Amadori-BV-BL}]
	\label{l:26}
	There exist a positive constant $ h_1 $ such that for all $ h\in (0,h_1) $ and for all $ \ul,\ur \in\brz$, there exists a unique $ h $-Riemann solver. More precisely, there exist $ n+2 $ states $ \ul=\omega_0,\om_1,...,\om_n,\om_{n+1}=\ur $ and $ n $ wave amplitudes $ \sigma_1,...,\sigma_n $, depending smoothly on $ \ul,\ur $, such that
	\begin{itemize}
		\item[(1)] $\om_i=\Psi_i[\sigma_i](\om_{i-1})$,\qquad $ i=1,...,m,m+2,...,n+1 $;
		\item[(2)] $ \om_{m+1}=\phu{\om_m} $.
	\end{itemize}
\end{lemma}

\subsubsection{Construction of $ \eh $-solutions}
\label{sss:cehs}

First, following \cite{Amadori-BV-BL} we give the definition of $ \eh $-solutions.
\begin{definition}\label{d:eh-es}
	For any given time $ T>0 $ and any fixed $\e>0$, we say that a continuous map
	\[
	t\mapsto \ueh(t,\cdot)\in \lnorm 1(0,L), \quad \forall t\in (0,T)
	\]
	is an \emph{$(\e,h)$-solution} to quasilinear hyperbolic system of balance laws \eqref{e:gs} on the domain $ \D_T $ if 
	\begin{itemize}
		\item As a function of two variables, $\ueh=\ueh(t,x)$ is piecewise constant with discontinuities occurring along finitely many straight lines which are called \emph{fronts}.
		\item Restricted on the domain $ Z'_j:=\{(j-1)h<x<jh,\ 0<t<T\}\cap \D_T $ $ (j=1,..., [L/h]+1) $, where $ [a] $ denotes the integer part of $ a $, $ \ueh $ is an $ \e $-solution to the homogeneous system of conservation laws \eqref{e:hcl}.
		\item  Along the segment $ \mmz'_j: =\{0<t<T\}\times\{x=jh\}\ (j=1,..,[L/h])$, the values $ \ul=\ueh(t,jh-) $ and $ \ur=\ueh(t,jh+) $ satisfy $ \ur=\phu{\ul} $ for all $ t\in (0,T) $ except at the points where fronts interact with the segment $ \mmz'_j $, and we also call the discontinuity with zero speed the zero-wave.
	\end{itemize}	
	Moreover, if the initial and boundary values of $\uhe$ satisfy approximately the initial-boundary conditions \eqref{e:ibc}, namely, 
	\begin{subequations}
		\label{d:ueh-aib}
		\begin{gather}
		\|\ueh(0,\cdot)-\bar u(\cdot)\|_{\lnorm 1(0,L)}\leq \e,\label{d:ueh-ai}\\
		\|b_1\big(\ueh(\cdot,0+)\big)-g_1(\cdot)\|_{\lnorm{1}(0,T)}\leq \e,\quad \|b_2\big(\ueh(\cdot,L-)\big)-g_2(\cdot)\|_{\lnorm{1}(0,T)}\leq \e, \label{d:ueh-ab}
		\end{gather}
	\end{subequations}
	then $\ueh=\ueh(t,x)$ is called the \emph{$\eh$-solution} to the mixed initial-boundary value problem \eqref{e:gs}-\eqref{e:ibc}.
\end{definition}

From the definition, we see that for $\eh$-solutions, besides rarefaction fronts, shocks and contact discontinuities, there are also \emph{zero waves} which travel with the zero speed.

For any given positive constants $ \e$ and $h $ and for any given initial-boundary data $ (\iu,\go,\gt) $ with $ \lug $ sufficiently small, as in the construction of $ \he $-solutions, we first choose piecewise constant approximate initial-boundary data $ (\iue,\goe,\gte) $ satisfying \eqref{h:ib-app}, then the piecewise constant approximate solution $ u=\ueh(t,x) $ can be constructed for sufficiently small $ t $ as follows.
\begin{itemize}
	\item At each point $ x=jh $ for $ j\in (0,L/h)\cap \Z $, we apply the \emph{approximate $ h $-Riemann solver} which just replaces rarefaction waves in the solution generated by the $ h $-Riemann solver, by $ m=\left[\sigma /\e\right]+1$ fronts, each of which has amplitude $ \sigma/m\le\e $ as in the approximate Riemann solver for homogeneous system \eqref{e:hcl}. We denote by $ \mz $ the set of zero-waves.
	\item At the jump point of $ \iue $ outside $ \mz $ and at boundaries $ x=0$ and $ x=L $, we use a classical Riemann solver for system\eqref{e:hcl} (see \cite{Li-Yu_OC}). 
\end{itemize}


All fronts travel with constant speed until they interact with each other or hit the boundary.  Here, by a slight modification of the speed of physical fronts, we may assume that no more than two fronts interact at an inner point and no more than one front hits the boundary at a time. At every interaction point, a new Riemann problem arises as in the case of $ \e $-solutions. In order to prevent that the number of fronts becomes infinite in a finite time, a \emph{simplified approximate Riemann solver} has been proposed when the corresponding interaction strength is below a certain threshold value $ \rho $. When one front interacts with the boundary or two fronts interact at  a point $ (\tau,\xi)\in  Z'_j $ for some $ j\in (0,T/h)\cap \Z $, we use the classical approximate or simplified approximate Riemann solver for the homogeneous system \eqref{e:hcl} as in \cite{Li-Yu_OC}.  When a front $ \alpha\in \P \cup\NP$ with amplitude $ \sa $ interacts with the front $ \beta \in\mz $ with amplitude $ \sb $ at a given point $ (\tau,\xi) $, without loss of generality, we may assume that the front $ \alpha $ is located on the left of $ \beta $, then the following different situations should be considered for some prescribed constant $ \rho $ which will be specified later:

\begin{itemize}
	\item If $ \alpha\in \P $ and $ |\sa\sb|\ge \rho $, we use the approximate $ h $-Riemann solver.
	\item If $ \alpha\in \P $ and $ |\sa\sb|<\rho $, we use the \emph{simplified approximate $ h $-Riemann solver} defined as follows: let $ \ul $, $ \um=\Psi_{\ka}[\sa](\ul) $ and $ \ur=\phu{\um} $ be the states before the interaction. Taking
	\[
	\tum=\Phi_h(\ul),\qquad \tur=\Psi_{\ka}[\sa](\tum),
	\]
	the interaction generates three fronts: a zero-front $ [\ul,\tum] $, a physical front $ [\tum,\tur] $ and a non-physical front $ [\tur,\ur] $.
	\item If  $ \alpha\in \NP $, we use a \emph{crude  $ h $-Riemann solver} defined as follows: let $ \ul,\ \um $ and $ \ur=\phu{\um} $ be the states before the interaction. Defining $ \tum=\phu{\ul} $, two fronts are generated by the interaction: a zero-front $ [\ul,\tum] $ and a non-physical front $ [\tum,\ur] $.
\end{itemize}

This algorithm is well-defined if we can prove that the total amplitude of fronts remains sufficiently small and the total number of fronts can be controlled. 
The details of the proof can be found in Appendix. In fact, we can prove the following

\begin{proposition}
	\label{p:ueh-ex}
	For any fixed $T>0$, there exist positive constants $\delta$ and $C(T)$ such that for every initial-boundary data $(\iu,\go,\gt)$ with
	\begin{equation*}
	\lug+\gamma L \le \delta,
	\end{equation*}
where $ \gamma  $ is defined in Theorem \ref{t:cl}, and for each $\e>0,\ h>0$, there exists an $\eh$-solution  $\ueh=\ueh(t,x)$ to the initial-boundary value problem \eqref{e:gs}-\eqref{e:ibc} for all $t \in (0,T)$, satisfying
	\begin{subequations}
		\begin{align}
		\tv{\ueh(t,\cdot)}{0<x<L}+\gamma L & \le C(T) ( \lug+\gamma L), \quad \forall t\in (0,T), \label{bd:ueh-tvx}\\
		\tv{\ueh(\cdot,x)}{0 <t<T} &\le C(T) \left(\lug+\gamma L\right),\quad \forall  x\in (0,L),\label{bd:ueh-tvt}\\
		\|\ueh(t,\cdot)-\ueh(s,\cdot)\|_{\lnorm 1(0,L)} &\leq C(T)  |t-s|,\quad \forall t,s \in (0,T),	\label{bd:ueh-L1t}\\
		\|\ueh(\cdot,x)-\ueh(\cdot,x')\|_{\lnorm 1(0,T)} &\leq C(T) |x-x'|,\quad \forall x,x' \in (0,L).\label{bd:ueh-L1x}
		\end{align}
	\end{subequations}
\end{proposition}

Moreover, for a fix pair of $ (\e,h) $, we have the approxiamte stability for $ \eh $-solutions.   
\begin{figure}
	\begin{minipage}[t]{0.5\linewidth} 
		\centering
		\includegraphics[width=6.3cm]{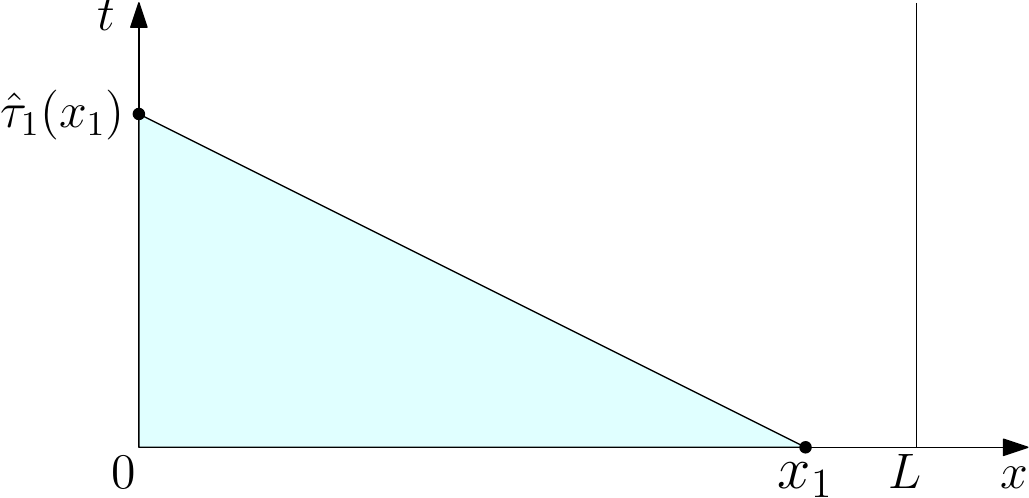}
		\caption{triangular domain $ \mfl'(x_1) $}
		\label{f:lt}
	\end{minipage}%
	\begin{minipage}[t]{0.5\linewidth}
		\centering
		\includegraphics[width=6cm]{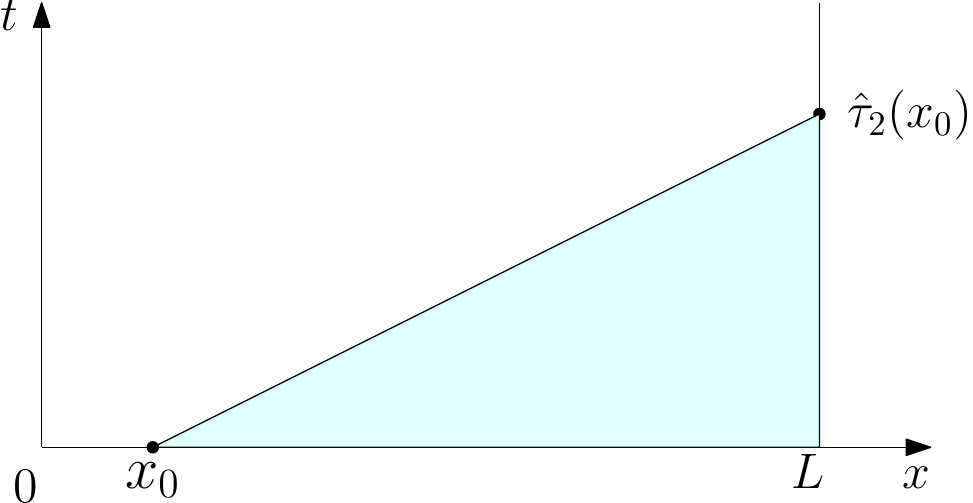}
		\caption{triangular domain $\mfr'(x_0)$}
		\label{f:rt}
	\end{minipage}
\end{figure}
\begin{proposition}
	\label{p:ueh-stab}
	For any given $T>0$, let $\ueh,\veh$ be two $\eh$-solutions to problem \eqref{e:gs}-\eqref{e:ibc} on the domain $\{0<t<T,0<x<L\}$, given by Proposition \ref{p:ueh-ex}. There exists a positive constant $ C $ independent of the choice of $ x_1\in (0,L] $ and $x_0\in [0,L)$, such that
	\begin{equation}
	\label{est:ueh-sl}
	\begin{split}
	&\|\ueh(t,\cdot)-\veh(t,\cdot)\|_{\lnorm 1(\mflt(x_1))} \leq  C \bigg(\|\ueh(0,\cdot)-\veh(0,\cdot)\|_{\lnorm 1(0,x_1)} \\
	&\qquad \qquad + \int^{t}_0 |b_1(\ueh(s,0+))-b_1(\veh(s,0+))|ds +\e t \bigg), \quad \forall t\in[0,\hatt(x_1)],
	\end{split}
	\end{equation}
	where $\hat{\tau}_1(x_1)=x_1\min_{u\in\brz}\{|\lambda_1(u)|^{-1}\}$ and 
	\[\mfl_t(x_1)=\{x\ |\ 0<x<x_1(\hat{\tau}_1(x_1)-t)/\hat{\tau}_1(x_1)\},
	\]
	and
	\begin{equation}
	\label{est:ueh-sr}
	\begin{split}
	&\|\ueh(t,\cdot)-\veh(t,\cdot)\|_{\lnorm 1(\mathfrak{R}_t(x_0))} \leq  C \bigg(\|\ueh(0,\cdot)-\veh(0,\cdot)\|_{\lnorm 1(x_0,L)} \\
	& \qquad \qquad \qquad + \int^{t}_0 |b_2(\ueh(s,L-))-b_2(\veh(s,L-))|ds +\e t \bigg),\quad \forall t\in [0,\hatt(x_0)],
	\end{split}
	\end{equation}
	where  $\hat{\tau}_2(x_0)=x_0\min_{u\in\brz}\{|\lambda_n(u)|^{-1}\}$ and 
	\[\mfr_t(x_0)=\{x\ |\ (L-x_0)t/\hatt_2(x_0)+x_0<x<L \},
	\]
	and there exists a positive constant $ C(T) $ depending on time $ T $, such that
	\begin{equation}
	\label{est:ueh-s}
	\begin{split}
	&\|\ueh(t,\cdot)-\veh(t,\cdot)\|_{\lnorm 1(0,L)}\\
	\leq & C(T)\bigg(\|\ueh(0,\cdot)-\veh(0,\cdot)\|_{\lnorm 1(0,L)} + \int^{t}_{0} \Big(\big|b_1(\ueh(s,0+))-b_1(\veh(s,0+))\big|\\
	& \qquad \qquad + \big|b_2(\ueh(s,L-))-b_2(\veh(s,L-))\big|\Big) ds +\e \bigg), \quad \forall t\in (0,T).
	\end{split}
	\end{equation}
\end{proposition}

\begin{remark}
	\label{r:ueh-u}
	Estimates \eqref{est:ueh-sl} and \eqref{est:ueh-sr} concern the approximate stability of $ \eh $-solutions on the triangle domains $ \mfl'(x_1):=\{ (t,x)\ |\ 0<t<\hatt_1(x_1),\  0<x<x_1(\hatt_1(x_1)-t)/\hatt_1(x_1) \} $ (see Fig. \ref{f:lt}) and $ \mfr'(x_0):=\{ (t,x)\ |\ 0<t<\hatt_2(x_0),\ (L-x_0)t/\hatt_2(x_0)+x_0<x<L \} $ (see Fig. \ref{f:rt}), respectively. This two estimates are not only useful in proving the estimate \eqref{est:ueh-s} by induction, but also needed to show the determined domain of limit solutions to the mixed initial-boundary value problem with one-sided boundary.
\end{remark}

\subsubsection{The limit of $ \eh $-solutions}

Now we fix $ h>0 $ and choose a sequence  $ (\enu,h) $ with $\enu\to 0$ as $ \nu\to +\infty $. For each $\nu\ge 1$, Proposition \ref{p:ueh-ex} yields an $(\enu,h)$-solution $\unuh$ to problem \eqref{e:gs}-\eqref{e:ibc}, such that for all $ \nu\ge 1 $ the maps $t \mapsto \unuh(t,\cdot)$ are uniformly Lipschitz continuous in $ \lnorm 1 $ norm with respect to $ t $, and $ \displaystyle{\tv{\unuh(t,\cdot)}{0<x<L} }$ remains sufficiently small for all $ t\in (0,T) $. By Helly's Theorem, we can extract a subsequence of $ \{\unuh\} $ which converges to a limit function $\uh=\uh(t,x)$ in $\lnorm 1((0,T)\times(0,L))$ as $ \nu\to +\infty $. Moreover, by Proposition \ref{p:ueh-stab} the whole sequence $ \{\unuh\} $ is a Cauchy sequence and converges to a unique limit as $ \nu\to +\infty $.  In fact, for every $ \nu_1,\nu_2\ge 1 $, it is easy to see that $ u^{\nu_2,h}$ (resp. $ u^{\nu_1,h}$) is also an $(\e_{\nu_1},h)  $-solution (resp. $(\e_{\nu_2},h)  $-solution) if $ \nu_2> \nu_1 $ (resp. if $ \nu_1>\nu_2 $), then by  \eqref{est:ueh-s} we have
\begin{equation}\label{e:ueh-cvg}
\begin{split}
&\|\unuoh(t,\cdot)-\unuth(t,\cdot)\|_{\lnorm 1(0,L)}\\
\leq & C(T)\bigg(\|\unuoh(0,\cdot)-\unuth(0,\cdot)\|_{\lnorm 1(0,L)} + \int^{t}_{0} \Big(\big|b_1(\unuoh(s,0+))-b_1(\unuth(s,0+))\big|\\
& \qquad \qquad + \big|b_2(\unuoh(s,L-))-b_2(\unuth(s,L-))\big|\Big) ds +\max\{\e_{\nu_1},\e_{\nu_2}\} \bigg).
\end{split}
\end{equation}
From \eqref{d:ueh-ai}-\eqref{d:ueh-ab}, we know that the right-hand side of \eqref{e:ueh-cvg} approaches to zero as $ \nu_1,\nu_2\to + \infty $.

One can prove that, by using the same proof of Theorem 3 in \cite{Amadori-BV-BL}, the limit $ \uh=\uh(t,x) $ satisfies the system
\begin{equation*}
\pt H(u)+\px F(u)=\sum_{j\in\mz\cap (0,T/h)} G(u)\delta_0(x-jh)h
\end{equation*}
on $ \D_T $ in the sense of distributions.  Indeed, we have the following

\begin{proposition}
	\label{p:uhs}
	For any given $ h>0 $ sufficiently small, let $\{ \uehn \}$ be a sequence of $ \eh $-solutions to problem \eqref{e:gs}-\eqref{e:ibc} with $ \enu\to 0 $ as $ \nu\to +\infty $. Then $ \uehn $ converges as $ \nu\to +\infty $ to a unique limit $ \uh $ in $ \L^1_{loc}(\D_T) $, which satisfies for any given $ \phi\in C^1_c(\D_T) $,
	\begin{equation}\label{e:uh}	
	\int^{T}_0\int^L_0\left[H(u)\phi_t+F(u)\phi_x\right]+\int^T_0 \bigg[\sum_{j\in\Z\cap (0,L/h)} \phi(t,jh)
	G(u(t,jh-))h \bigg]dt=0.
	\end{equation}  
\end{proposition}
\noindent Moreover, by \eqref{h:ib-app} and \eqref{bd:ueh-L1t}-\eqref{bd:ueh-L1x}, we can conclude that $ \uh $ also satisfies the initial-boundary conditions \eqref{e:ibc}.

\vspace{6pt}

Since the estimates \eqref{bd:ueh-tvx} and \eqref{bd:ueh-L1t} are independent of $ \eh $, the corresponding estimates still uniformly hold for the sequence of $\{\uh\}$. By choosing a sequence $ \{h_i\} $ with $ h_i\to 0 $ as $ i\to  +\infty $, we can apply Helly's compactness theorem to $\{\uhi\} $ to extract a convergent subsequence (still denoted by $ \uhi $) converging to some function $ u\in \L^1(\D_T) $. By a standard argument (see Theorem 4 of \cite{Amadori-BV-BL}), we can prove that $ u=u(t,x) $ is indeed an entropy solution to problem \eqref{e:gs}-\eqref{e:ibc}. In fact, we have
\begin{proposition}
	\label{p:uhex}
	Let $ \{u^{h_i}\} $ be a subsequence of solutions to the mixed initial-boundary value problem \eqref{e:uh} and \eqref{e:ibc} with uniformly small total variation, converging in $ \L^1(\D_T)$ to some function $ u=u(t,x) $ as $ i\to +\infty $. Then $ u $ is an entropy solution to problem \eqref{e:gs}-\eqref{e:ibc}.
\end{proposition}

We need to show that this convergence is independent of choice of  subsequences. Roughly speaking, this can be done by the following steps:
\begin{enumerate}[(1)]
	\item Recall that the notations \eqref{d:lt}-\eqref{d:rt} of trapezoid domains $\mfl$ and $ \mfr $, and write 
\begin{equation}\label{d:mflr}
		 \mflt=\mflt(L) ,\quad \mfr_t=\mfr_t(0).
\end{equation} 
	Let $ \tu =u\cdot\chi_{\mfl} $ be the solution restricted on the domain $ \mfl $, where $ \chi_{\mfl} $ is the characteristic function of $ \mfl $. We can show that $ \tu $ satisfies some additional so called ``viscosity solution" properties (c.f. \cite{Bressan2000}) on the time interval $ (0,\hatt) $. On the other hand, if $ \hat u=\hat u(t,\cdot) $ is a Lipschitz continuous map with values in $ \L^{1}((0,L);\rn n) $ satisfying $\hat u(t,x)\equiv 0 $ on $ (0,L)\setminus \mflt $ and these ``viscosity solution" properties, then $ \hat u $ coincides with $ \tu $ on the time interval $ (0,\hatt) $. This implies that, restricted on $ \mfl $, all convergent subsequences of $ \{\uhi\}$ converge to a unique limit in $ \L^1(\mfl) $.  
	\item Similarly, it can be proved that, restricted on $ \mfr $, all convergent subsequences of $ \{\uhi\}$ converge to the same limit in $ \L^1 (\mfr)$. Therefore, $ \uhi$ converges to the same limit in $ \L^1(\D_{\hatt}) $, since $ \D_{\hatt}\subseteq \mfl\cup \mfr  $.
	\item By a similar argument, we can prove that $ \uhi$ converges to a unique limit  in $ \L^1((j\hatt,(j+1)\hatt)\cap \D_T)\times (0,L) )$ for all $ j\in \Z\cap[1,T/\hatt] $. This implies that $ \uhi$ converges to the same limit $ u $ in $ \L^1(\D_T) $.
\end{enumerate}

In what follows, we only prove Step (1), and 
Steps (2)-(3) can be proved in an analogous way. Since all the technical proofs are quite similar to those in \cite{Amadori-BV-BL}, in what follows, we give the statement of the main results, and the sketch of the proofs can be found in Appendix.

For any $ t\in [0,\hatt) $, we denote by $ {\md}^*_t $ the set of pair of functions $ (u,g) $ with $ u:(0,L)\to \rn{n} $ and $ g:(t,\hatt)\to \rn{n-m}, $ such that $ u(x)\equiv 0 $ on $ (0,L)\setminus \mflt $, then by Proposition \ref{p:ueh-stab} we have the following
\begin{proposition}
	\label{p:uniq}
	There exist a family of sets $ \{\md_t \}_{t\in (0,\hatt)} $ with $ \md_t  \subseteq {\md}^*_t$ and a process 
	\[
	P(t,t_0): \md_{t_0}\to \md_{t+t_0}\quad \text{for all }\ t_0,t,t+t_0\in (0,\hatt)
	\] 
	with the following properties:
	\begin{enumerate}[(i)]
		\item for all $ U\in \md_{t_0} $, $ P(0,t_0)U=U$, $ P(t,t_0+s)\circ P(s,t_0)U=P(t+s,t_0)U $;
		\item if $ U=(u,g)\in \md_{t_0} $, then $ P(t,t_0)U=\big(E(t,t_0)U,\T(t,t_0) g\big) $, where $\T(t,t_0)g$ restricts $ g $ on $ (t+t_0,\hatt) $, 
		and the operator $ E(t,t_0) $ is defined on $ \md_{t_0} $, such that for any $ t_0,t_0'\in [0,\hatt), \ t\in [0,\hatt_1-t_0], \ t'\in [0,\hatt-t_0'] $ and for each $ U'=(u',g')\in \md_{t'} $ and $ U''=(u'',g'')\in \md_{t''} $, we have 
		\[
		\begin{split}
		\llnorm{E(t',t'_0)U'-E(t'',t''_0)U''} \le &C\bigg(\llnorm{u'-u''}+|t'-t''|+|t'_0-t''_0|\\
		&+\int_{t_0}^{t_0+t}\|g'(\tau)-g''(\tau)\|d\tau\bigg);
		\end{split}
		\]
		\item There exists a sequence of $ \{  P^{h_i} \}$, such that $ P^{h_i}(t,t_0) U $ converges to $ P(t,t_0)U $ as $ i \to +\infty $ for any given $ U\in \md_{t_0} $ with respect to the distance $d(P^{h_i}(t,t_0)U,P(t,t_0)U)  $, where
		\[
		d(U',U'')=\|\iu'-\iu''\|_{\L^1(0,L)}+\|g'-g''\|_{L^1(t,\hatt)}
		\]
		for any $ U'=(\iu',g'),\ U''=(\iu'',g'')\in\md_t $;
		\item for all $\bar{U} \in \md_0 $, the function $ u(t,\cdot)=E(t,0)\bar{U} $ is an entropy solution of system \eqref{e:gs}.     
	\end{enumerate}
\end{proposition}

In order to use Proposition \ref{p:uniq} to prove the uniqueness of the limit of any sequence of $ \{\uh\} $, we need to introduce some notations: for any given function $ u=u(t,x) $ and any given point $ (\tau,\xi) \in \D_T $, if $ \xi>0 $, we define $ U^{\sharp}_{(u;\tau,\xi)} $ to be the solution to the Riemann initial value problem
\[
\begin{cases}
H(\omega)_t+F(\omega)_x=0,\\
\omega(\tau,x)=
\begin{cases}
u(\tau,\xi-), & \text{if}\ x<\xi,\\
u(\tau,\xi+), & \text{if}\ x>\xi
\end{cases}
\end{cases}
\]
for all $ t>\tau $; while, if $ \xi=0 $, we define $ U^{\sharp}_{u;\tau,\xi} $ to be the solution to the Riemann initial-boundary value problem
\[
\begin{cases}
H(\omega)_t+F(\omega)_x=0,\\
\omega(\tau,x)=u(\tau,\xi+), & \text{if}\ x>0,\\
b(\omega(t,0))=g(\tau+), & \text{if}\ t>\tau
\end{cases}
\]
for all $ t>\tau $.

Let $ U^{\flat}_{(u;\tau,\xi)} $ be the solution to the Cauchy problem of the linear hyperbolic system with constant coefficients
\[
\begin{cases}
\tilde{A} w_t+\tilde{B}w_x=\tilde{G},\\
w(\tau,x)=u(\tau,x)
\end{cases}
\]
for $ t>\tau $, where $ \tilde{A}=DH(u(\tau,\xi)) $, $ \tilde{B}=DF(u(\tau,\xi)) $ and $ \tilde{G}=G(u(\tau,\xi)) $.

\begin{proposition}
	\label{p:vs}
	Let $ \hat{\lambda} >0$ be an upper bound of all characteristic speeds. Then every trajectory $ u(t,\cdot)=E(t,0)\bar{U}, \bar{U}=(\iu,g)\in \md_0 $, satisfies the following conditions at each $ \tau\in (0,T) $: 
	\begin{enumerate}[(1)]
		\item If $ \xi>0 $, then we have
		\begin{equation}\label{e:urs-iv}
		\lim_{\theta \to 0}\frac{1}{\theta}\int\limits_{\xi-\theta \hat{\lambda}}^{\xi+\theta\hat{\lambda}}\left|u(\tau+\theta,x)-U^{\sharp}_{u;\tau,\xi}(\theta,x)\right|dx=0,
		\end{equation}
		and for every $ \rho >0$ sufficiently small and $ \theta $ with $ 0<\theta<\rho/\hat{\lambda} $, we have
		\begin{equation}\label{e:ules-in}
		\begin{split}
		\frac{1}{\theta}\int_{\xi-\rho+\theta\hat{\lambda}}^{\xi+\rho-\theta\hat{\lambda}}&\left|u(\tau+\theta,x)-U^{\flat}_{u;\tau,\xi}(\tau+\theta,x)\right|dx\\
		&\le C\Big[ \tv{u(\tau,\cdot)}{\xi-\rho<x<\xi+\rho}+2\rho\Big].
		\end{split}
		\end{equation}
		\item If $ \xi=0 $, then we have
		\begin{equation}\label{e:urs-bv}
		\lim_{\theta \to 0}\frac{1}{\theta}\int\limits_{0}^{\theta\hat{\lambda}}\left|u(\tau+\theta,x)-U^{\sharp}_{u;\tau,\xi}(\theta,x)\right|dx=0.
		\end{equation}
	\end{enumerate}
	
	Vice versa, let $ u=u(t,x):[0,\hatt]\to \L^1((0,L);\rn{n})  $ be a Lipschitz continuous map such that $ u(0)=\iu $, $ u(t,x)\equiv 0 $ on $ \D_{\hatt}\setminus \mfl $ and $ b_1(u)=g $ on $ x=0 $ for some given $ (\iu,g)\in \md_0 $.  
	Assume that conditions \eqref{e:urs-iv}-\eqref{e:urs-bv} hold, then $ u $ coincides with a trajectory of E, that is, for all $ t\in (0,\hatt) $ we have
	\begin{equation*}
	u(t,\cdot)=E(t,0)(\iu,g).
	\end{equation*}
\end{proposition}

The proof can be obtained by combining the corresponding results in \cite{Amadori_viscosity} and \cite{Amadori-BV-BL}. In fact, Theorem 8.1 in \cite{Amadori-BV-BL} concerns \eqref{e:urs-iv}-\eqref{e:ules-in} for the Cauchy problem of balance laws \eqref{e:gs}. On the other hand, since the zero-waves can not give any impact on the boundary, the proof of Theorem 5.2 in \cite{Amadori_viscosity} concerning \eqref{e:urs-bv}, for the initial-boundary value problem of conservation laws \eqref{e:hcl}, can be applied here with a little modification.

Propositions \ref{p:uniq}-\ref{p:vs} imply that the convergence of a sequence $ \{u^{h_i}\} $ stated in Proposition \ref{p:uhex} is independent of the choice of subsequences. Then, together with Propositions \ref{p:ueh-ex}-\ref{p:ueh-stab} and Proposition \ref{p:uhs}, we have
\begin{theorem}
	\label{t:lueh}
	For any fixed $T>0$, there exist positive constants $\delta$ and $ C(T) $ such that for every initial-boundary data $(\iu,g^u_1,g^u_2)$ with
	\begin{equation*}
	\Lambda(\iu,g^u_1,g^u_2) +\gamma L\le \delta,
	\end{equation*}
	problem \eqref{e:gs}-\eqref{e:ibc} associated with the initial-boundary data $ (\iu,g^u_1,g^u_2) $ admits a solution $ u=u(t,x) $ on the domain $ \D_T $ as the limit of $\eh$-solutions given by Proposition \ref{p:ueh-ex}, which satisfies
	\begin{equation*}
	\label{bd:lueh-tv}
	\begin{split}
	\tv{u(t,\cdot)}{0<x<L}+\gamma L \leq  C(T) \Lambda(\iu,g^u_1,g^u_2), \quad \forall t\in (0,T),\\
	\|u(t,\cdot)-u(s,\cdot)\|_{\lnorm 1(0,L)}\leq C(T) |t-s|,\quad \forall t,s \in (0,T),
	\end{split}
	\end{equation*}
	and $u(t,x)\in \brz$ for a.e. $(t,x)\in \D_T  $.
	
	Moreover, suppose that $ v=v(t,x) $ is a solution as the limit of $\eh$-solutions to problem \eqref{e:gs}-\eqref{e:ibc} associated with the initial-boundary data $(\iv,g^v_1,g^v_2)$ with $ \Lambda(\iv,g^v_1,g^v_2)+\gamma h\le \delta $ given by Proposition \ref{p:ueh-ex}. Then, for any given $ x_0\in [0,L) $ and $ x_1\in (0,L] $, there exist a positive constant $ C $ independent of $ x_0 $ and $ x_1 $, such that
	\begin{equation*}
	\label{est:lueh-lstab}
	\begin{split}
	&\|u(t,\cdot)-v(t,\cdot)\|_{\lnorm 1(\mflt(x_1))} \\
	\leq &C \left(\|\iu-\iv\|_{\lnorm 1(0,x_1)} +  \int^{t}_0 |g^u_1(s)-g^v_1(s))|ds \right),\quad\forall t\in [0,\hat{\tau}_1(x_1)],
	\end{split}
	\end{equation*}
	\begin{equation*}
	\label{est:lueh-rstab}
	\begin{split}
	&\|u(t,\cdot)-v(t,\cdot)\|_{\lnorm 1(\mathfrak{R}_t(x_0))}\\
	\leq & C \left(\|\iu(0,\cdot)-\iv(0,\cdot)\|_{\lnorm 1(x_0,L)} + \int^{t}_0 |g^u_2(s)-g^v_2(s)|ds \right),\quad \forall t\in[0,\hat\tau_2(x_0)],
	\end{split}
	\end{equation*}
	and there exists a positive constant $ C(T) $ depending on $ T $, such that
	\begin{equation*}
	\label{e:stability}
	\begin{split}
	&\|u(t,\cdot)-v(t,\cdot)\|_{\lnorm 1(0,L)}\\
	\leq & C(T)  \bigg(\|\iu-\iv\|_{\lnorm 1(0,L)} + \sum_{i=1,2}\int^{t}_{0} \big|g^u_i(s)-g^v_i(s)\big|ds \bigg),\quad \forall t\in (0,T).
	\end{split}
	\end{equation*}
\end{theorem}

\subsection{Relation between $ \he $-solutions and $ \eh $-solutions}
\label{ss:rbhesehs}

In \cite{Li-Yu_OC}, under the assumption that all negative eigenvalues are linearly degenerate, we proved the equivalence between two $ \e $-solutions to the forward problem and to the corresponding rightward problem for system \eqref{e:hcl}, respectively. In fact, we have the following

\begin{lemma}\label{l:ebes}
	Let system \eqref{e:hcl} satisfy assumptions (H1)-(H4). Suppose that all the negative eigenvalues are linearly degenerate. Suppose furthermore that $ \ue $ is an $ \e $-solution to system \eqref{e:hcl} in the forward sense. Then, if we exchange the role of $ t $ and $ x $, namely, regard $ x $ as the ``time" variable and $ t $ as the ``space" variable, $ \ue $ is also an $ \e $-solution to system \eqref{e:hcl} in the rightward sense, i.e. an $ \e $-solution to system
	\begin{equation}\label{e:hcl-r}
	\px F(u)+ \pt H(u)=0.
	\end{equation}
	And vice versa.
\end{lemma}

Using Lemma \ref{l:ebes} and comparing Definition \ref{d:he-es} and Definition \ref{d:eh-es}, we can easily  prove the following lemma.
\begin{lemma}
	\label{l:uhe-cv}
	Let system \eqref{e:gs} satisfy assumptions (H1)-(H4). Suppose that $\uhe(t,x)$ is a $\he$-solution to system \eqref{e:gs} in the forward sense on the domain $\D_T$. Then, if we exchange the role of $ t $ and $ x $, $ \uhe $ is also an $\eh$-solution to system \eqref{e:gs} in the rightward sense, that is, it is also an $ \eh $-solution to system
	\begin{equation}
	\label{e:rs}
	\px F(u)+ \pt H(u)=G(u)
	\end{equation}
	on the domain $ \D_T $. The analogous results also hold vice versa.
\end{lemma}
\begin{proof}
	Suppose that $ \uhe $ is a $ \he $-solution to system \eqref{e:gs} in the forward sense. By Definition \ref{d:he-es}, restricted on each set $Z_j:=\{jh<t<(j+1)h,\ 0<x<L\}$ for $ j\in \mz $, $ \uhe $ is an $ \e $-solution to system \eqref{e:hcl}. Then by Lemma \ref{l:ebes}, $ \uhe $ is also an $ \e $-solution to system \eqref{e:hcl-r} in the rightward sense on $ Z_j $, which verifies (2) of Definition \ref{d:eh-es}. Moreover, along each segment $ z_j $, $ \uhe(jh-,x)=H^{-1}[H(\uhe(jh-,x)+hG(\uhe(jh-,x))] $ for $ x\in (0,L) $ except all interaction points, this immediately implies that $ \uhe $ satisfies (3) of Definition \ref{d:eh-es} in the rightward sense. 
\end{proof}

By passing to the limit, we can easily obtain the corresponding equivalent results for the limit solutions. In fact, we can still apply Helly's theorem to $ \he$-solutions in the rightward sense by estimates \eqref{bd:uhe-tvt} and  \eqref{bd:uhe-L1conti-x} and to $ \eh $-solutions in the forward sense by estimates \eqref{bd:ueh-tvt} and \eqref{bd:ueh-L1x}.
\begin{proposition}
	\label{p:frs}
	Under the same assumptions of Lemma \ref{l:uhe-cv}, if $ u $, as the limit of $ \he $-solutions given by Theorem \ref{t:luhe-ex}, is a solution to problem \eqref{e:gs}-\eqref{e:ibc} in the forward sense, then $ u $ is also a solution  as the limit of $\eh$-solutions to system \eqref{e:rs} in the rightward sense. Similar results hold from the solution as the limit of $ \eh $-solutions to system \eqref{e:rs} in the rightward sense to that as the limit of $ \he $-solutions to system \eqref{e:gs} in the forward sense.
\end{proposition}

Let $ \displaystyle T^*\geq L\max_{u\in \brz}\frac{1}{|\lambda_m(u)|} $. Suppose that $ u=u(t,x) $ is a solution as the limit of $ \he $-solutions to problem \eqref{e:gs}-\eqref{e:ibc} on $\D_{T^*} =\{0<t<T^*,\ 0<x<L\} $. Then by Proposition \ref{p:ueh-stab}, Proposition \ref{p:frs} and Remark \ref{r:ueh-u}, $ u $ is also the unique solution as the limit of $ \eh $-solutions to system \eqref{e:hcl-r} in the rightward sense on the triangle domain $ \mfl^*:= \{0<t<T^*,\ 0<x<L(T^*-t)/T^*\}  $. Therefore, we have the following
\begin{proposition}
	\label{p:lueh-t}
	Under the same assumptions as in Lemma \ref{l:uhe-cv}, assume that $ u=u(t,x) $ is a solution to problem \eqref{e:gs}-\eqref{e:ibc} in the forward sense on the domain $ \D_{T^*} $, given by Proposition \ref{p:uhe-ex}. Then on the triangular domain $ \mfl^* $, $ u $ coincides with the solution $ \tilde u $ as the limit of $\eh$-solutions to system \eqref{e:rs} associated with the initial condition
	\begin{equation*}
	\label{ic:artifical-rightward}
	x=0:\ \tu=u(\cdot,0+)
	\end{equation*}
	and the following boundary condition corresponding to the original initial data $ \iu $:
	\begin{equation*}
	\label{bc:artifical-rightward}
	t=0:\ \tilde b_1(\tu)=\tilde b_1(\iu),
	\end{equation*}
	where $ \tilde b_1  \in \C1(B_r(0)$; $\rn{n-m})$ is an arbitrarily given function satisfying the same assumption \eqref{h:bc} as $ b_1 $.
\end{proposition}

\section{Local exact one-sided boundary null controllability}
We are now ready to apply the results obtained in the previous sections for semi-global solutions to prove Theorem \ref{t:cl}, namely, to realize the local exact one-sided boundary null controllability for a class of general hyperbolic systems of balance laws
\label{s:oneside}
\begin{equation}
\label{e:forward}
\pt H(u)+ \px F(u)=G(u),\qquad t>0,\ 0<x<L
\end{equation}
under the assumption that all the negative characteristic fields are linearly degenerate.

In order to establish Theorem \ref{t:cl}, it suffices to prove the following
\begin{lemma}
	\label{lemma:one-side-bc}
	Under the same assumptions as in Theorem \ref{t:cl}, for any given initial data $\iu$ with $\displaystyle{\tv{\iu}{0<x<L}+|\iu(0+)|}$ sufficiently small, system \eqref{e:forward} together with the boundary condition
	\begin{equation*}
	\label{bc:reduce-oneside-IBVP}
	x=0:\ b_1(u)=0, \qquad t\in (0,T)
	\end{equation*}
	admits a solution $ u=u(t,x) $ on the domain $\{\ 0< t< T,\ 0< x< L\}$ with small $ \displaystyle \tv{u(\cdot,L-)}{0<t<T} +u(0,L-)$, satisfying simultaneously the initial condition \eqref{ic:intro} and the final condition \eqref{e:fc}.
\end{lemma}

In fact, let $ u=u(t,x) $ be a solution given by Lemma \ref{lemma:one-side-bc}. Taking the boundary control as
\[
g_2(t):=b_2(u(t,L-)),\quad \forall t\in(0,T),
\]
we obtain the local exact boundary null controllability desired by Theorem \ref{t:cl}.

	{\em Proof of Lemma \ref{lemma:one-side-bc}.} If \eqref{d:time-cl} holds, then for $ r>0 $ sufficiently small, we have
	\begin{equation}
	\label{e:time-control-brz}
	T>L \max_{u\in \brz}\left\{\frac{1}{|\lambda_m(u)|}+\frac{1}{\lambda_{m+1}(u)}\right\}.
	\end{equation}
	Step 1. Let
	\begin{equation}
	\label{d:To}
	T_1:=L\max_{u\in \brz}\frac{1}{|\lambda_m(u)|}.
	\end{equation}
	Choosing an artificial function $g_f$ with $\displaystyle\tv{g_f}{0< t< T_1}+|g_f(0+)|$ sufficiently small, we consider the forward problem of system \eqref{e:forward} with the following initial condition and boundary conditions:
	\begin{eqnarray*}
	t=0: && u=\iu,\\
	x=0: && b_1(u)=0,\\
	x=L: && b_2(u)=g_f.
	\end{eqnarray*}
	By Theorem \ref{t:luhe-ex}, there exists a solution $u_f=u_f(t,x)$ obtained as the limit of $\he$-solutions on the domain $D_{T_1}=\{0<t<T_1,0<x<L\}$ with 
	$u_f(t,x)\in \brz$ for $ (t,x)\in D_{T_1} $ (see Figure \ref{f:f}).
	
	\begin{figure}[H]
		\begin{minipage}[t]{0.42\linewidth} 
			\centering
			\includegraphics[width=0.5\linewidth]{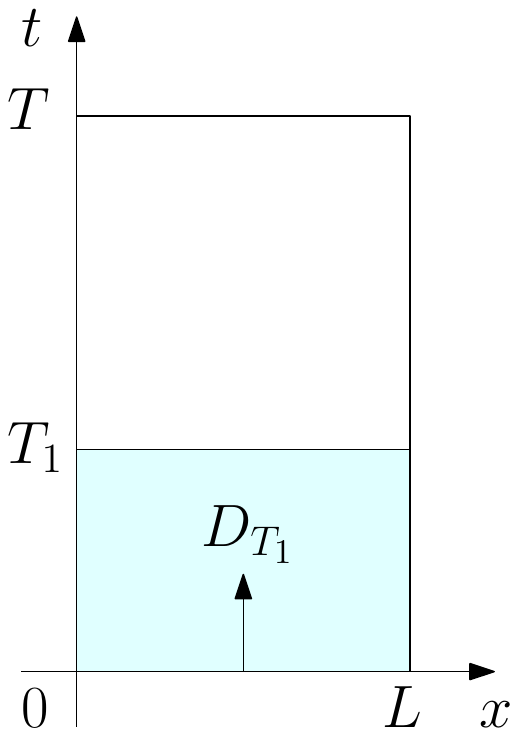}
			\caption{forward solution}
			\label{f:f}
		\end{minipage}%
		\qquad\qquad
		\begin{minipage}[t]{0.42\linewidth}
			\centering
			\includegraphics[width=0.5\linewidth]{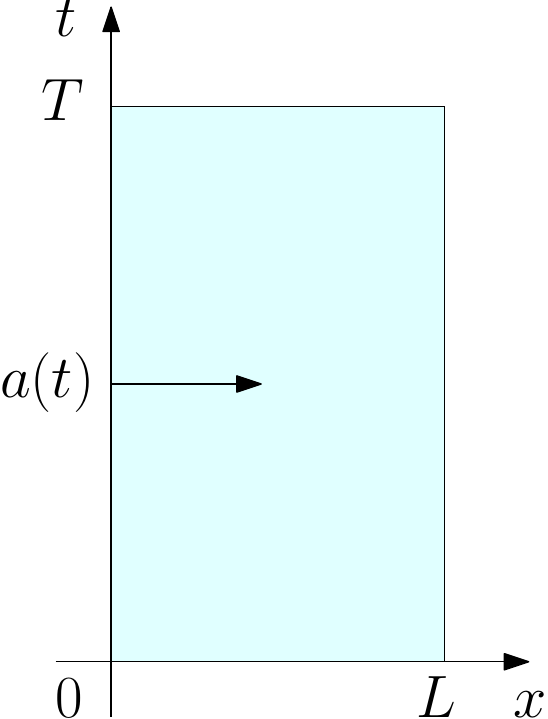}
			\caption{rightward solution}
			\label{f:b}
		\end{minipage}
	\end{figure}
	
	Step 2. Let
	\[
	a(t)=\begin{cases}
	u_f(t,0+) & 0<t<T_1,\\
	0 & T_1<t<T	.	
	\end{cases}
	\]
	Obviously, $ a(t)\in \brz $  with sufficiently small total variation.
	
Now we change the role of variables $t$ and $x$ and consider the rightward problem (see Figure \ref{f:b}) for the system
\begin{equation*}
\px F(u)+\pt H(u)=G(u), \qquad 0<x<L,\ 0<t<T
\end{equation*}
with the initial condition
\begin{equation*}
x=0:\ u=a(t),\quad 0<t<T
\end{equation*}
and the following boundary conditions reduced from the initial state $u=\iu$ and the final state $u=0$:
\begin{eqnarray}
t=0: && l_s(u) u=l_s(\iu)\iu, \quad s=m+1,...n,\nonumber 
\\
t=T: && l_{r}(u) u=0, \quad r=1,...,m, \label{bc:a-r}
\end{eqnarray}
where $l_i(u)\ (i=1,...,n)$ are the left eigenvectors of $(DH(u)^{-1}DF(u) $, or equivalently, the left eigenvectors of $ (DF(u))^{-1}DH(u) $. A direct computation shows that this boundary condition satisfies the corresponding assumption (H5).
	
By Theorem \ref{t:lueh}, the rightward problem admits a solution $u=u(t,x)$ on the domain $ \D_T $ as the limit of $ \eh $-solutions. By Proposition \ref{p:frs}, the function $ u $ is a solution to system \eqref{e:forward} in the forward sense on $ \D_T $ and it is an entropy solution as far as the system possesses an entropy-entropy flux. Since $ u(t,0)=a(t) $ for a.e. $ t\in (0,T) $, we have
	\[
	b_1(u(t,0+))=0, \qquad \text{a.e.}\ t\in (0,T).
	\]	
	
Step 3. It now remains to show that $u$ verifies the initial condition \eqref{ic:intro} and the final condition \eqref{e:fc}.
	
By Proposition \ref{p:frs}, both $ u_f $ and $ u $ are solutions in the rightward sense. Then by Proposition \ref{p:lueh-t} for the rightward problem, and by \eqref{d:To},  $u_f$ coincides with $ u $ on the triangular domain $\{0\leq  t \le T_1,\  0\leq x \le L(T_1-t)/T_1\}$ (see Figure \ref{f:it}). This implies the initial condition \eqref{ic:intro}.
	
Since $ u(t,0)\equiv 0 $ for $ T_1\leq t\le T$ and $ u $ satisfies \eqref{bc:a-r} for $ 0\le x\le L $, by \eqref{e:time-control-brz}-\eqref{d:To} for the rightward problem, according to Theorem \ref{t:lueh}, we have $ u(t,x)\equiv 0 $ on the triangular domain $\left\{T_1\le t\le T,\ 0\le x\le L\left({t-T_1})/({T-T_1})\right)\right\}$ (see Figure \ref{f:bt}). In particular, we get the final condition \eqref{e:fc}.
	
Thus $ u=u(t,x) $ is a desired solution and the proof of Lemma \ref{lemma:one-side-bc} is complete.
\begin{figure}[H]
	\begin{minipage}[t]{0.42\linewidth} 
		\centering
		\includegraphics[width=0.53\linewidth]{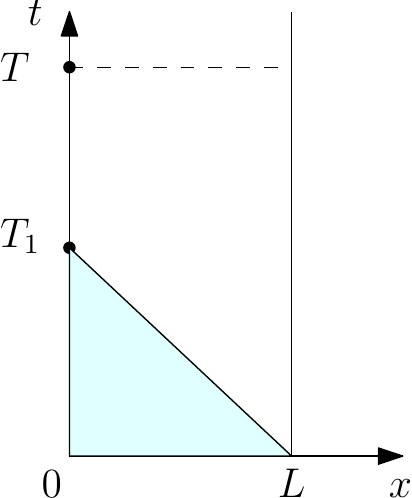}
		\caption{Uniqueness for the initial-boundary value problem with one-sided boundary conditions in the rightward sense}
		\label{f:it}
	\end{minipage}%
	\qquad\qquad
	\begin{minipage}[t]{0.42\linewidth}
		\centering
		\includegraphics[width=0.53\linewidth]{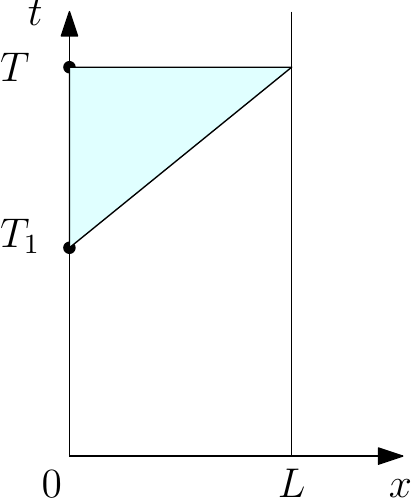}
		\caption{Uniqueness of null solution on the triangle domain}
		\label{f:bt}
	\end{minipage}
\end{figure}

\section{Appendix}
Expecting that the readers are familiar with standard front-tracking algorithm for constructing approximate solutions to hyperbolic systems of conservation laws (for example Chapter 7 of \cite{Bressan2000}), we collect all the technical proofs in this appendix.

\subsection{Proof for Proposition \ref{p:uhe-ex}}
\label{ss:puhe}
Supposing that $ \alpha $ is a front and $ x=\xat $ denotes its location at time $ t $, we recall that the notation $ \sa(t)
$, or just $\sigma_{\alpha}$ for shortness, denotes the amplitude of
front $ \alpha $. Let $ \Gamma_1 $ be the set of time when one of the following situations occurs:
\begin{itemize}
	\item[(a)] two fronts interact with each other inside the domain $\D_T$;
	\item[(b)] one front hits the boundary;
	\item[(c)] the boundary data $\gie\ (i=1\text{ or }2)$ has a jump.
\end{itemize}
Let $ \Gamma_2 $ be the set of time $ \{t= jh:\ j\in \Z\cap(0, T/h]\} $.

We estimate the wave amplitude on the trapezoid domain $ \mfl $ (see \eqref{d:lt}). For this purpose, we define Glimm-type functionals restricted on 
\begin{equation*}
\mflt:=\left\{x\ |\ 0<x<L(\hat \tau-t)/{(2\hat \tau)}\right\}
\end{equation*}
as follows. For all $ t\in (0,\hat{\tau})\setminus(\Gamma_1\cup\Gamma_2) $ we first define
\begin{equation}
\label{def:Vel}
V^{L}(\uhe(t,\cdot)):=\sum_{0<\xa(t)<L(1-t/\hat\tau)}K^L_{\alpha}|\sa(t)|+C_1\tv{g_1^{\e}(s)}{t<s<T},
\end{equation}
in which the sum takes over all fronts cross the segment $ \mflt $, 
\[
K^L_{\alpha}=
\begin{cases}
K, & \text{if }  \mathcal{F}(\alpha)\le m, \\
0, & \text{if }  \mathcal{F}(\alpha)\ge m+1,
\end{cases}
\]
and $ K $ and $ C_1 $ are positive constants to be specified later. Let
\begin{equation*}
\label{def:Qel}
Q^L(\uhe(t,\cdot)):=\sum_{\substack{0<\xa(t),\xb(t)<L(1-t/\hat\tau)\\(\alpha,\beta)\in \mathcal{A}}} |\sa(t)\sb(t)|
\end{equation*}
be the functional measuring the wave interaction potential, where $\mathcal{A}$ is the set of all approaching waves, which may contain non-physical fronts. More precisely, supposing that a front $ \alpha $ is located on the left of a front $ \beta $, we say that $ \alpha,\beta $ are approaching waves if and only if {$ \F(\a)>\F(\b) $, or $ \F(\a)=\F(\b) $ and at least one of $ \a $ and $ \b $ is a shock. Let
	\begin{equation*}
	\label{def:Uel}
	\Upsilon^L(\uhe(t,\cdot)):=V^{L}(\uhe(t,\cdot))+C_2 Q^{L}(\uhe(t,\cdot)),
	\end{equation*}
	where $ C_2 $ is a positive constant to be specified later.
	
	In a similar way, we can do the same thing on the triangular domain $ \mfr $ (see \eqref{d:rt})   .  For this purpose, we define the corresponding Glimm functionals $ V^{R} $, $ Q^{R} $ and $ \Upsilon^{R} $ for $ \uhe(t,\cdot) $ restricted on  $\mathfrak{R}_t:=\left\{x\ |\ Lt/(2\hat\tau)<x<L\right\}$ as follows: 
	\begin{gather}
	V^{R}(\uhe(t,\cdot)):=\sum_{Lt/\hat{\tau}<\xa(t)<L}K^R_{\alpha}|\sa(t)|+C_1\tv{g_1^{\e}(s)}{t<s<T},\label{def:Ver}\\
	Q^R(\uhe(t,\cdot)):=\sum_{\substack{Lt/\hat\tau<\xa(t),\xb(t)<L \\ (\alpha,\beta)\in \mathcal{A}}} |\sa(t)\sb(t)|,\label{def:Qer}\\
	\Upsilon^R(\uhe(t,\cdot)):=V^{R}(\uhe(t,\cdot))+C_2 Q^{R}(\uhe(t,\cdot)),\label{def:Uer}
	\end{gather}
	where 
	\[
	K^R_{\alpha}=
	\begin{cases}
	0, & \text{if } \mathcal{F}(\alpha)\le m, \\
	K, & \text{if } \mathcal{F}(\alpha)\ge m+1,
	\end{cases}
	\]
	and $ K $ and $ C_i\ (i=1,2) $ are positive constants to be specified later.
	
	By \eqref{eq:2-1}-\eqref{eq:1} and a simliar argument used in Section 5.4 of \cite{Li-Yu_OC}, after choosing the constants $ K, C_1 $ and $ C_2 $ properly, it can be proved that 
	\begin{equation}\label{e:uhe-gamma1}
	\Upsilon^L(t+)-\Upsilon^L(t-)<0,\ \Upsilon^R(t+)-\Upsilon^R(t-)<0 \quad \text{for all $ t\in \Gamma_1\cap(0,\hat{\tau})$}.
	\end{equation}
	In fact, setting
	\[
	\Delta V^L(\uhe(t,\cdot)):=V^L(\uhe(t+,\cdot))-V^L(\uhe(t-,\cdot)),
	\]
	and defining $\Delta Q^L(\uhe(t,\cdot))$ and $ \Delta \Upsilon^L(\uhe(t,\cdot))$ in a similar way, we can check (\ref{e:uhe-gamma1}) in the following three cases:
	\begin{enumerate}[(1)]
		\item when two fronts $\alpha$ and $\beta$ interact at time $t=\tau$ inside $\mfl$, by \eqref{eq:2-1}-\eqref{eq:2-2} we have
		\begin{align*}
		&\Delta V^{L}(\uhe(\tau+,\cdot))\le C K|\sigma_{\alpha}\sigma_{\beta}|,\\
		&\Delta Q^L(\uhe(\tau,\cdot))\le -|\sigma_{\alpha}\sigma_{\beta}|+CKV^L(\uhe(\tau-,\cdot))|\sigma_{\alpha}\sigma_{\beta}|\le -\frac{1}{2}|\sigma_{\alpha}\sigma_{\beta}|,
		\end{align*}
		provided that $V^L(\uhe(\tau-,\cdot))$ is sufficiently small. These immediately imply that 
		\begin{equation}
		\label{e:u-1}
		\Delta \Upsilon^L(\uhe(\tau,\cdot))\le -\frac{1}{4}|\sigma_{\alpha}\sigma_{\beta}|,
		\end{equation}
		provided that $C_2\ge 2CK+1/2$.
		\item When a front $\alpha$ hits the left boundary $x=0$, by (\ref{eq:4}) we have
		\begin{align*}
		\Delta V^L(\uhe(\tau,\cdot))\le (C-K)|\sigma_{\alpha}|,\quad \Delta Q^L(\uhe(\tau,\cdot))\le C|\sigma_{\alpha}|V^L(\uhe(\tau-,\cdot)),
		\end{align*}
		then
		\begin{equation}
		\label{e:u-2}
		\Delta \Upsilon^L(\uhe(\tau,\cdot))\le -\frac{1}{2}|\sigma_{\alpha}|,
		\end{equation}
		provided that $K\ge \max\{4C,1\}$ and $V^L(\uhe(\tau,\cdot))$ is sufficiently small. 
		\item If $g^{\e}_1(\tau+)\ne g^{\e}_1(\tau-)$, then by \eqref{eq:1} we have
		\begin{align*}
		&\Delta V^L(\uhe(\tau,\cdot))\le (C-C_1)|\Delta g^{\e}_1(\tau)|,\\
		&\Delta Q^L(\uhe(\tau,\cdot))\le C |\Delta g^{\e}_{1}|V^L(\ueh(\tau-,\cdot))\le \frac{1}{2}|\Delta g^{\e}_1(\tau)|,
		\end{align*}
		provided that $C_1\ge C+1$ and $V^L(\uhe(\tau-,\cdot))$ is sufficiently small. These immediately imply that
		\begin{equation}
		\label{e:u-3}
		\Delta \Upsilon^L(\uhe(\tau,\cdot))\le -\frac{1}{2}|\Delta g^{\e}_1(\tau)|.
		\end{equation}
	\end{enumerate}
	Thus, $\Upsilon^L(\uhe(t,\cdot))$ decreases across time $t=\tau$ for all $\tau\in \Gamma_1$. A result can be obtained for $\Upsilon^R(\uhe(t,\cdot))$ in a similar way.
	
	In order to estimate the jump of $ \Upsilon^L(\ueh(t,\cdot))$ and $\Upsilon^R(\uhe(t,\cdot)) $ at time $t\in \Gamma_2 $, we need the following lemma, which can be obtained by a similar proof as in Lemma 2.1 in \cite{Amadori-unique_bl}.
	\begin{lemma}\label{l:uhe-gamma2}
		Let $ \uhe $ be an $ \he $-solution constructed in Section \ref{sss:ches}. For each $ t\in \Gamma_2 \cap (0,\hat{\tau}) $, we have
		\begin{gather*}
		\Upsilon^L(\uhe(t+,\cdot))\le (1+Ch)\Upsilon^L(\uhe(t-,\cdot)),\\
		\Upsilon^R(\uhe(t+,\cdot))\le (1+Ch)\Upsilon^R(\uhe(t-,\cdot)).
		\end{gather*}
	\end{lemma}
	
	It is easy to see that these functionals mentioned above can only change their values across time $ t\in \Gamma_1\cup \Gamma_2  $. Then by \eqref{e:uhe-gamma1} and Lemma \ref{l:uhe-gamma2}, we have
	\begin{equation*}
	\Upsilon^L(\uhe(t,\cdot))\le (1+Ch)^k \Upsilon^L(\uhe(0+,\cdot))
	\end{equation*}
	for all $ t\in (kh,(k+1)h)\cap(0,\hat{\tau}) $ with $ k\in \Z\cap (0,\hat{\tau}/h) $. Noticing that $ (1+Cs)^k\le e^{Cks}\le e^{Ct} $, we have
	\begin{equation}\label{est:Upl}
	\Upsilon^L(\uhe(t,\cdot))\le e^{Ct} \Upsilon^L(\uhe(0+,\cdot)),\quad \forall t\in (0,\hat{\tau}).
	\end{equation}
	
	Similarly, we can prove that
	\begin{equation}\label{est:Upr}
	\upr(\uhe(t,\cdot))\le e^{Ct} \upr(\uhe(0+,\cdot)),\quad \forall t\in (0,\hat{\tau}).
	\end{equation}
	
	Now for any given $ t\in (0,T] $, we recall the standard Glimm-type functional without weight $ K^L_{\alpha}$ or $ K^R_{\alpha} $:
	\begin{align*}
	&V(\uhe(t,\cdot)):=\sum_{0<\xa(t)<L}|\sa(t)|+C_1\sum_{i=1,2}\tv{g_i^{\e}(s)}{t<s<T},\\
	&Q(\uhe(t,\cdot)):=\sum_{\substack{0<\xa(t),\xb(t)<L \\ (\alpha,\beta)\in \mathcal{A}}} |\sa(t)\sb(t)|,\\
	&\Upsilon(\uhe(t,\cdot)):=V_{\tau}(\uhe(t,\cdot))+C_2 Q(\uhe(t,\cdot)).
	\end{align*}
	Since $(0,L)=\mflt\cup \mathfrak{R}_t$ for all $t\in [0,\hatt]$, we have
	\begin{equation}\label{e:uulur}
	\Upsilon(\uhe(t,\cdot)) \le (\upl+\upr)(\uhe(t,\cdot))\le 2K \Upsilon (\uhe(t,\cdot)),\quad t\in [0,\hatt].
	\end{equation}
	Then, combining \eqref{est:Upl}, \eqref{est:Upr} and \eqref{e:uulur}, we get
	\begin{equation*}
	\Upsilon(\uhe(t,\cdot))\le C(t) \Upsilon(\uhe(0+,\cdot))\quad \forall t\in(0,\hatt].
	\end{equation*}
	Repeating the same argument on the time interval $ [k\hat{\tau},(k+1)\hat{\tau}] $ for $ k\in (0,T/\hat{\tau}) $, we have
	\begin{equation*}
	\Upsilon(\uhe(t,\cdot))\le C(t) \Upsilon(\uhe(k\hatt,\cdot)),\quad \forall t\in(k\hatt,(k+1)\hatt]\cap(0,T)
	\end{equation*}
	for all $ k \in\Z\cap (0,T/\hatt)$. Therefore, by induction we have
	\begin{equation}\label{est:up}
	\Upsilon(\uhe(t,\cdot))\le C(t) \Upsilon(\uhe(0+,\cdot))\quad \forall t\in(0,T).
	\end{equation}
	
	Observing that $Q(\uhe(t,\cdot))$ is always positive, by \eqref{est:up} we have
	\begin{equation}
	\label{est:vup}
	V(\uhe(t,\cdot))
	\le \Upsilon(\uhe(t,\cdot))
	\le e^{Ct}\Upsilon(\uhe(0+,\cdot)).
	\end{equation}
	Since $Q(\uhe(0+,\cdot))\le V(\uhe(0+,\cdot)))^2$, it follows from \eqref{est:vup} that
	\begin{equation*}
	\begin{split}
	\tv{\uhe(t,\cdot)}{0<x<L}
	& \le \oo V(\uhe(t,\cdot))
	\le \oo e^{Ct}\Upsilon(\uhe(0+,\cdot))\\
	&\le \oo e^{Ct} \big(V(\uhe(0+,\cdot))+C_2 (V(\uhe(0+,\cdot)))^2\big)\\
	&\le \oo e^{Ct}V(\uhe(0+,\cdot))
	\end{split}
	\end{equation*}
	for all $t\in (0,T)$, provided that $V(\uhe(0+,\cdot))$ is sufficiently small. Observing that 
	\[
	V(\uhe(0+,\cdot))\le \oo \Lambda(\iuhe,\gohe,\gthe)\le \oo \lug,
	\]
	we obtain estimate \eqref{bd:uhe-tvx}. 
	
	We now prove estimate \eqref{bd:ueh-tvt} under the assumption $ \uhe(t,x)\in \brz $ for all $ (t,x)\in \D_T $. It follows from \cite{Li-Yu_OC} that for each $ x\in (0,L) $,
	\begin{align*}
	&\tv{\uhe(\cdot,x)}{jh<t<(j+1)h}\le C \upl(\uhe(jh,\cdot)),  \quad j\in [0,(1-x/L)\hatt/h-1]\cap \Z, \\
	&\tv{\uhe(\cdot,x)}{jh<t<(j+1)h}\le C \upr(\uhe(jh,\cdot)),  \quad j\in [0,xT/(Lh)-1]\cap \Z
	\end{align*}  
	for all  $t\in (jh,(j+1)h)\cap(0,\hatt)$ with $ j\in [1,T/\hat{\tau}]\cap (0,T)$. This immediately implies that 
	\begin{equation}\label{e:uhetvtu}
	\tv{\uhe(\cdot,x)}{t\in(jh,(j+1)h)\cap(0,\hat{\tau})}\le C\Upsilon(\uhe(jh,\cdot)).
	\end{equation}
	Since, it is easy to see from \eqref{est:Upl}-\eqref{e:uulur} that
	\begin{equation}\label{e:Upuhe}
	\Upsilon(\uhe((j+1)h,\cdot))\le Ce^{h}\Upsilon(\uhe(jh,\cdot)),
	\end{equation}
	in order to obtain the estimate on $ \displaystyle \tv{\uhe(\cdot,x)}{0<t<\hatt} $ for each $ x\in (0,L) $, we need to estimate $ |\uhe(jk+,x)-\uhe(jh-,x)| $ for all $ j\in (1,{\hatt/h})\cap\Z $.
	In fact, we have
	\begin{equation*}\label{key}
	|\uhe(jh+,x)-\uhe(jh-,x)|= h|G(\uhe(jh-,x)| \le h\gamma,
	\end{equation*}
	then
	\begin{equation}\label{e:uhej}
	\sum_{ j\in (1,{\hatt/h})\cap\Z} |\uhe(jk+,x)-\uhe(jh-,x)|\le \gamma \hatt.
	\end{equation}
	Thus, by \eqref{e:uhetvtu}-\eqref{e:uhej}, it is easy to see that
	\begin{equation*}\label{key}
	\tv{\uhe(\cdot,x)}{0<t<\hatt} \le C(\hatt)\Upsilon(\uhe(0,\cdot))+\gamma \hatt.
	\end{equation*}
	
	Similarly we have
	\begin{equation*}\label{key}
	\tv{\uhe(\cdot,x)}{k\hatt<t<(k+1)\hatt}\le C(\hatt)\Upsilon(\uhe(k\hatt,\cdot))+\gamma \hatt
	\end{equation*}
	for each $ k\in[1,T/\hatt]\cap\Z$. Then, by induction and estimate \eqref{est:up} we obtain 
	\[
	\tv{\uhe(\cdot,x)}{0<t<T}\le C(T)\Upsilon(\uhe(0,\cdot))+\gamma T.
	\] 
	This leads to estimate \eqref{bd:uhe-tvt}. 
	
	Finally, we need to show that  $ \uhe(t,x)\in \brz$  for all $ (t,x)\in \D_T $. In fact, for any given $ (\tau,\xi)\in\D_T $, by estimates \eqref{bd:uhe-tvx}-\eqref{bd:ueh-tvt}, we have
	\begin{equation*}\label{key}
	\begin{split}
	|\uhe(\tau,\xi)&|\le |\uhe(0,0+)|+\tv{\uhe(0,\cdot)}{0<x<\xi}+\tv{\uhe(\cdot,\xi)}{0<t<\tau}\\
	&\le C(\delta+\gamma T)< r,
	\end{split}
	\end{equation*}
	provided that $ \delta $ and $ \gamma $ are sufficiently small.
	
	\subsection{Technical proofs for $ \eh $-solutions}
	The proofs of the following lemmas can be found in \cite{Amadori-BV-BL} with a small modification corresponding to our situation.
	\begin{lemma}
		\label{l:1}
		For any given $ u,u_1,u_2 \in \brz$ and for $ h $ sufficiently small, we have
		\begin{align*}
		\ab{\phu{u}-u}&\le\oo h|g(u)|\le \oo h\gamma,\\
		\ab{\phu{u_2}-\phu{u_1}-(u_2-u_1)}&\le \oo h\gamma\ab{u_2-u_1}.
		\end{align*}
	\end{lemma}
	
	In what follows, we write $ E $ as the implicit function given by Lemma \ref{l:26}:
	\[
	\sigma=E(h,\ul,\ur).
	\]
	
	\begin{lemma}
		\label{l:3}
		Let $ \sigma=E(h,\ul,\ur) $. Then we have
		\begin{equation*}\label{e:wsj}
		\begin{split}
		|\ul-\ur|&\le C \lrb{|\sigma|+h\ga},\\
		\ab{\sigma}&\le C\lrb{|\ul,\ur|+h\ga}.
		\end{split}
		\end{equation*}
	\end{lemma}	
	
	\begin{lemma}
		For any given $ u\in\brz$ and any given $\sigma\in \rn{n} $ and $ h>0 $ sufficiently small, we have
		\begin{equation}\label{e:cd}
		|\phu{\Psi[\sigma](u)}-\Psi[\sigma](\phu u)|\le C|\sigma|\ab{h\ga}.
		\end{equation}
	\end{lemma}

	\begin{lemma}
		\label{l:2}
		Suppose that $ r,\ h $ and $ |\sigma| $ are sufficiently small. For any given $ \umps,\vmps\in \brz $ such that
		\[
		\ups=F^{-1}[F(\ums+)+hG(\ums)],\quad \vms=\Psi[\sigma](\ums),\quad \vps=F^{-1}[F(\vms)+hG(\vms)],
		\]
		let $ \sigma=(\sigma_1,...,\sigma_n) $ and $  \sigma'=(\sigma'_1,...,\sigma'_n) $ be implicitly given by $\vms=\Psi(\sigma)(\ums)  $ and $\vps=\Psi(\sigma')(\ups)$ (or by $ \vms=S[\sigma](\ums)$ and $ \vps=S[\sigma'](\ups) $), respectively. Thus we have
		\[
		\sum\limits_{i=1}^{n}|\sigma_i-\sigma_i'|\le C h\sum\limits_{i=1}^{n}|\sigma_i|.
		\]

	\subsubsection{Proof of Proposition \ref{p:ueh-ex}}
	\label{ss:pueh}
		
	{\rm	Similarly to the case of $ \he $-solutions, as in \eqref{def:Vel}-\eqref{def:Uer} etc, we can propose Glimm-type functionals $ \Vl, \Ql, \upl,\Vr,\Qr $, $ \upr $, $ V,\ Q $ and $ \Upsilon $ for $ \ueh $ restricted on $ \mflt $ and $ \mr_t $, respectively. The only difference is that now the sum in \eqref{def:Vel} and \eqref{def:Ver} takes over all fronts including zero-waves $ \alpha\in\mz $, and the set $  \A$ of pairs of approaching fronts is replaced by an extended one $ \tilde{\A} $. In fact, we say that the couple of fronts $ (\alpha,\beta) $ with $ \xa<\xb $ belongs to $ \tilde{\A} $ if one of the following situations occurs:
		\begin{itemize}
			\item $ (\alpha,\beta) \in \A$ as in the homogeneous case;
			\item $ \alpha $ is a zero-wave and $ \beta $ is a physical one with $ \mathcal{F}(\beta)\le p $;
			\item $ \beta $ is a zero-front and $ \alpha $ is a physical one with $ \mathcal{F}(\alpha)>p $.
		\end{itemize}
		
		According to Lemmas \ref{l:1}-\ref{l:2}, one can prove, by a similar argument used in Section \ref{ss:puhe} for $\e$-solutions, that $ \upl(\ueh(t)) $ and $ \upr(\ueh(t)) $ is non-increasing for all $ t\in [0,\hat{\tau}) $. Moreover, $ \upl(\ueh(t)) $ and $ \upr(\ueh(t)) $ decrease at each interaction time $ \tau $, provided that $ \Vl(\ueh(\tau-)) $ remains sufficiently small. Then, for all $ t\in (0,\hatt) $, we have
		\begin{equation}
		\begin{split}
		\Vl(\ueh(t))\le \upl(\ueh(t))\le \upl(\ueh(0+))\le C \Vl(\ueh(0+)),\\
		\Vr(\ueh(t))\le \upl(\ueh(t))\le \upr(\ueh(0+))\le C \Vr(\ueh(0+)).
		\end{split}
		\label{e:vlr}
		\end{equation}
		By Lemma \ref{l:3}, there exits a positive constant $ C_1 $ such that
		\begin{equation}\label{e:tv-v}
		\begin{split}
		\frac{1}{C_1} \left\{ \tv{\ueh(t,\cdot)}{0<x<L}+\gamma L\right\}\le &V(\ueh(t))\\
		&\le C_1 \left\{ \tv{\ueh(t,\cdot)}{0<x<L}+\gamma L\right\}.
		\end{split}
		\end{equation}
		Thus, by a similar argument used for $ \he $-solutions (see also \cite{Amadori-BV-BL}), 
		for all $ t\in (0,\hatt) $ we get}
		\begin{equation}
		\label{e:tvb}
		\tv{\ueh(t,\cdot)}{0<x<L}+\gamma L \le C(t)\left[\tv{\ueh(0,\cdot)}{0<x<L}+\gamma L\right].
		\end{equation}
	\end{lemma}

	Moreover, all the arguments mentioned above can be easily applied to the time interval $ [j\hatt_1,(j+1)\hatt_1)\cap(0,T),\ \forall j\in \Z\cap (0,T/\hatt)$. Therefore, by induction, \eqref{e:tvb} holds for all $ t\in (0,T) $.
	
	Finally, we prove that $ \ueh(t,x)\in\brz $ for a.e. $ (t,x)\in \D_T $. As in the case of $ \he $-solutions, it suffices to prove \eqref{bd:ueh-tvt}. In fact, by classical interaction estimates between physical fronts and non-physical fronts, and the new interaction estimates \eqref{e:cd} involving zero-waves, a standard argument shows that for each $ x\in (0,L) $ we have
	\begin{align*}
	\tv{\ueh(\cdot,x)}{0<t<\hatt(L-x)/L}&\le C \Upsilon^L(\ueh(0+,\cdot)),\\
	\tv{\ueh(\cdot,x)}{\hatt /L<t<L} &\le C \Upsilon^R(\ueh(0+,\cdot)).
	\end{align*}   
	These yield \eqref{bd:ueh-tvt} as in the case of $ \he $-solutions.
	
	\subsubsection{Proof of Proposition \ref{p:ueh-stab}}
	
	\begin{lemma}[\cite{Colombo_general-balance-boundary}]
		\label{l:7}
		Suppose that $ u^*,v^*\in \brz $ and that there exist $ q^*_1,...,q^*_n $ in a small neighborhood of the origin in $ \R $, such that
		\begin{equation*}
		v^*=S_n(q^*_n)\circ\cdots\circ S_1(q^*_1)[u^*].
		\end{equation*}
		Then we have
		\begin{equation*}
		\sum_{j=m+1}^{n}|q^*_j|\leq \oo \sum_{i=1}^{m}\Big[|q^*_i|+|b_1(u^*)-b_1(v^*)|\Big]
		\end{equation*}
		and
		\begin{equation*}
		\sum_{i=1}^{m}|q^*_i|\leq \oo \sum_{j=m+1}^{n}\Big[|q^*_j|+|b_2(u^*)-b_2(v^*)|\Big].
		\end{equation*}
	\end{lemma}
	
	Using the resul ts obtained in \cite{Amadori-BV-BL} and
	\cite{Li-Yu_OC} on the stability of $\e$-solutions to the Cauchy problem of system \eqref{e:gs} and the stability of $\eh$-solutions to the initial-boundary value problem of system \eqref{e:hcl}, respectively, we can obtain estimate \eqref{est:ueh-sl}. For this purpose, equivalently to the $ \mathbb{L}^1 $-distance, we introduce a functional $ \Gamma=\Gamma(u,v) $ which is ``almost decreasing" along pairs of $ (u,v) $, where $u$ and $v$ are two $\eh$-solutions. At each point $(t,x)\in \mfl(x_1)$, we define the vector function $q=(q_1,..,q_n)$ implicitly by
	\[
	v(t,x)=S_n(q_n(t,x))\circ \cdots \circ S_1(q_1(t,x))[u(t,x)],
	\]
	and the intermediate states $U_0,U_1,...,U_n$ by
	\[
	U_0(t,x)=u(t,x),\quad U_i(t,x):=S_i(q_i (t,x))\circ \cdots \circ S_1(q_1(t,x))[u(t,x)],\quad  i=1,...,n.
	\]
	
	Let
	\[
	\lambda_i(t,x):=\lambda_i(\omega_{i-1}(t,x),\omega_{i}(t,x))
	\]
	be the speed of the $i$-shock connecting $U_{i-1}$ and $U_i$, and let
	\begin{equation}
	\label{d:tilde-q}
	\tqi=\begin{cases}
	\bar K q_i, & 1 \le i \le m,\\
	q_i,   & m+1 \le i \le n,
	\end{cases}
	\end{equation}
	where $\bar K$ is a positive constant to be specified later. Recalling the notation of $ \hat{\tau} $ in Proposition \ref{p:ueh-stab}, we define
	\begin{equation}
	\label{d:yo}
	y_1(t):= x_1(1-t/\hat \tau(x_1)), \qquad 0\leq t \leq \hat\tau(x_1).
	\end{equation}
	For all $ t\in (0,\hatt(x_1)) $, let $J_u(t)$ and $J_v(t)$ denote the set of physical fronts across the segment $ \{t\}\times (0,y_1(t)) $ in $u$ and $v$, respectively, and let $ \mz_u(t) $ and $ \mz_v(t) $ denote the set of zero-waves of $ u $ and $ v $, respectively.  Let $ J(t)=J_u(t)\cup J_v(t) $ and $ \mz(t)=\mz_u(t)\cup\mz_v(t) $. For each front $ \alpha\in J(t)\cup \mz(t) $,  $ x=\bxa(t) $ denotes its location. 
	Then we define
	\begin{equation}
	\label{def:wi}
	W_i(t,x)=1+\kappa_1 \sum_{\substack{\alpha\in J(t)\cup \mz(t)\\ 0<\xa<\yo(t)}} K_{i,\alpha}(x) |\sa(t) |+ \kappa_2\big( \upl(u(t))+\upl(v(t)) \big),\quad i=1,..,n,
	\end{equation}
	where $\kappa_1,\kappa_2$ are positive constants to be specified later (see \cite[Chapter 8]{Bressan2000}), and the weights $K_{i,\alpha}$ $ (\alpha\in J(t)\cup \mz(t)) $are given by
	\begin{equation*}
	K_{i,\alpha}(x)=
	\begin{cases} 
	1, & \text{if  }\alpha\in J(t),\ \ka >i\ \bxa<x,  \\
	1, & \text{if }\alpha\in J(t),\ \ka <i,\  \bxa>x,\\
	1, & \text{if } \alpha\in J(t),\ \ka =i,\ \text{$ i $-th characteristic is G.N.},\ q_i<0,\ \bxa<x,\\
	1, & \text{if } \alpha\in J(t),\ \ka =i,\ \text{$ i $-th characteristic is G.N.},\ q_i>0,\ \bxa>x,\\
	1, & \text{if } \alpha\in  \mz(t),\ i\le p,\ \bxa<x,\\
	1, & \text{if }\alpha\in  \mz(t),\ i> p,\ \bxa>x,\\
	0, &\text{otherwise.}
	\end{cases}		
	\end{equation*}
	
	Now, we consider the functional
	\begin{equation*}
	\Gamma(u(t),v(t))=\sum^n_{i=1}\int^{\yo(t)}_{0} |\tqi(t,x)| W_i(t,x) dx,
	\end{equation*}
	where $ y_1(t) $ is given by \eqref{d:yo}. As we always assume that the total variation of the approximate solution is sufficiently small, we have  $1\leq \wi(t,x) \leq 2$ $ (i=1,...,n) $. Then there exists a positive constant $\bar c$ such that
	\begin{equation}
	\label{est:equiv-Gamma-L1}
	\frac{1}{\bar c}\|u(t,\cdot)-v(t,\cdot)\|_{\lnorm 1(\mflt)}\leq \Gamma(u(t),v(t))\le \bar c \|u(t,\cdot)-v(t,\cdot)\|_{\lnorm 1(\mflt)}.
	\end{equation}
	
	Let $ \mathcal{J}\subset (0,T) $ be the set of time $ t $ when an inner/boundary interaction occurs in $ u $ or $ v $, or 
	when $ b_1(u(t,0+)) $ or $ b_1(v(t,0+) $ has a jump. 
	It is easy to see that $ t\mapsto \Gamma(u(t),v(t))  $ is Lipschitz continuous on any given time interval which does not contain points of $ \mathcal{J} $. Moreover, for any given  $t \in\mathcal{J} $, once the constant $k_1$ is fixed, we can choose $\kappa_2$ so large that the term $\kappa_2\big(\Upsilon^L(u(t))+\Upsilon^L(v(t))\big)$ can compensate the possible increase of the term $\displaystyle  \kappa_1 \sum_{\alpha\in J(t)\cup\mz(t)} K_{i,\alpha}(x) |\sa(t) |$ in \eqref{def:wi} when there is a jump of approximate boundary data $b_1(u(t,0+))$ or $ b_1(v(t,0+)) $, or an inner/boundary interaction. Therefore, we have
	\begin{equation*}
	\Gamma(u(t+),v(t+))-\Gamma(u(t-),v(t-))\leq 0,\quad  \forall t\in \mathcal{J}.
	\end{equation*}
	
	We claim that for a.e. $t\in (0,T)\setminus\mathcal{J} $, we have
	\begin{equation*}
	\frac{d}{dt}\Gamma(u(t),v(t))\leq  \tilde{r}(\e).
	\end{equation*}
	In fact, assume that at time $t$ there is no interaction, then we have
	\begin{equation}
	\label{e:diff-Gamma}
	\begin{split}
	&\quad\frac{d}{dt}\Gamma(u(t),v(t))\\
	&=\sum_{\alpha\in J(t)} \sum^n_{i=1}\big[\wami |\tqami| - \wapi |\tqapi| \big] \dot{\bar x}_{\alpha}(t )\\
	&-\sum^n_{i=1} \wyomi |\tqyomi| \dot y_1(t ).
	\end{split}
	\end{equation}
	
	Let $ \qbpi \ (i=1,...,n)$ be defined implicitly by
	\[
	v(t,0+)=S_n(q^{b}_n(t))\circ\cdots\circ S_1(q^{b}_1(t))[u(t,0+)],
	\]
	and let $ \tqbpi$ be the weighted $\qbpi $ similarly given by formula \eqref{d:tilde-q} for $ i=1,...,n $.
	
	Since $ u $ and $ v $ are piecewise constant, for all $ \alpha\in J(t) $ we have
	\[
	W_i(t,\bar{x}_{\alpha-1}(t)+)\tilde q_i(t,\bar{x}_{\alpha-1}(t)+)|\lambda_i(t,\bar{x}_{\alpha-1}(t)+)|  =\wami|\tqami|\lami
	\]
	and 
	\[
	W_i(t,\bar{x}_{1}(t)-)|\tilde q_i(t,\bar{x}_{1}(t)-)|\lambda_i(t,\bar{x}_{1}(t)-)
	=\wbpi|\tqbpi|\lbpi.
	\]
	
	Let $ \bar \alpha\in J(t)$ be the jump point which is
	closest to $ x=y_1(t) $. 
	We have
	\[
	| W_i(t,\bar{x}_{\bar{\alpha}}(t)+)\tilde q_i(t,\bar{x}_{\bar{\alpha}}(t)+)|\lambda_i(t,\bar{x}_{\bar{\alpha}}(t)+)=\wyomi|\tqyomi|\lyomi.
	\]
	
	Therefore, \eqref{e:diff-Gamma} can be rewritten as
	\begin{equation*}
	\label{est:decay-et}
	\begin{split}
	&\ \frac{d}{dt}\Gamma(u,v)(t )\\
	=&\sum_{\alpha\in J(t)\cup\mz(t)} \sum^n_{i=1}\Big[ \wapi |\tqapi|(\lapi-\dbxa(t)) \\
	&\qquad \qquad - \wami |\tqami| (\lami-\dbxa(t)) \Big]\\
	& + \sumin \wbpi |\tqbpi| \lbpi-\sumin \wyomi |\tqyomi| (\lyomi-\dot y_1(t))\\
	:=& \sum_{\alpha\in J(t)\cup\mz(t)} \sumin E_{\alpha,i}(t) + E_{b,i}(t) + E_{y_1,i}(t).
	\end{split}
	\end{equation*}
	As in \cite{Amadori-BV-BL}, for system \eqref{e:hcl}, by choosing $\kappa_1$ suitably large we have
	\begin{align}
	&\sumin E_{\alpha,i}(t) \le C\e|\sa|, \quad \qquad \alpha\in J(t),\label{eq:10}	
	\end{align}
	Moreover, it is proved in \cite{Amadori-BV-BL} that
	\begin{equation*}
	\label{eq:3}
	\sumin E_{\alpha,i}(t)\le C|\sa|,\quad \qquad \alpha\in \mz(t).
	\end{equation*}
	
	Concerning the term on the boundary, assumption \eqref{h:nonzero-char} implies $\lbpi\leq -c$. Since $\wbpi\ge 1$, by Lemma \ref{l:7} we have
	\begin{eqnarray*}
	\label{est:eb}
	E_{b,i}(t)
	&\leq & -c\bar K \sum_{i=1}^m |\tqbpi|+\oo\sum_{i=m+1}^n|\tqbpi|\nonumber \\
	&\leq & (\oo-c\bar K) \sum_{i=1}^m |\tqbpi| +\oo |b_1(u(t ,0+))-b_1(v(t ,0+))|\nonumber\\
	&\leq & \oo|b_1(u(t,0+))-b_1(v(t,0+))|,
	\end{eqnarray*}
	provided that $\bar K\ge \oo/c$.
	
	Observing that by the definition of $\hat\tau(x_1)$ in Proposition \ref{p:ueh-stab} and \eqref{d:yo}, we infer that $\dot y_1(t)< \lyomi$ for all $i=1,...,n$, which means that no front can enter the domain $ \mfl(x_1) $ through the line $ x=y_1(t)  $. This implies that
	\begin{equation}
	\label{est:eyo}
	E_{y_1,i}(t)\leq 0.
	\end{equation}
	
	Thus, the combination of \eqref{eq:10} and \eqref{est:eyo} yields 
	\begin{equation}
	\label{est:decay-Gamma}
	\frac{d}{dt}\Gamma(u(t),v(t))\leq \oo \big(\e+|b_1(u(t,0+))-b_1(v(t,0+))|\big)
	\end{equation}
	for a.e. $t\in (0,T)\setminus \mathcal{J}$.
	
	Once the constant $k_1$ is fixed, by \eqref{e:u-1}-\eqref{e:u-3} we can choose $\kappa_2$ so large that the term $\kappa_2\big(\upl(u(t))+\upl(v(t))\big)$ can compensate any possible increase of the term $\displaystyle  \kappa_1 \sum_{\alpha\in J(t)\cup\mz(t)} K_{i,\alpha}(x) |\sa(t) |$ in \eqref{def:wi} when there is a jump of approximate boundary data $b_1(u(t,0+))$ or $ b_1(v(t,0+)) $, or an inner/boundary interaction. Thus, for all $ t,s $ with $0 \le s<t\leq \hat\tau(x_1)$, by \eqref{est:decay-Gamma} we have
	\begin{equation*}
	\Gamma(u(t),v(t))-\Gamma(u(s),v(s))\leq \oo \left( \e|t-s|+ \int_s^t |b_1(u(\tau,0+))-b_1(v(\tau,0+))|d\tau \right).
	\end{equation*}
	Finally, the above inequality together with \eqref{est:equiv-Gamma-L1} yields the approximate stability estimate \eqref{est:ueh-sl} on the triangular domain $\mfl'(x_1)$.
	By a similar argument, we can get estimate \eqref{est:ueh-sr} on the triangular $\mfr'(x_0)$. Noting that $\hatt=\min\{ \hatt_1(L),\hatt_2(0) \}$ and $(0,L)=\mfl_t(L)\cup \mfr_t(0)$ for $t\in [0,\hatt)$, and combining \eqref{est:ueh-sl} and \eqref{est:ueh-sr}, for all $t\in [0,\hatt)$ we have
	\begin{eqnarray*}
	&&\|\uhe(t,\cdot)-\vhe(t,\cdot)\|_{\lnorm 1(0,L)}\\
	&\leq & \|\uhe(t,\cdot)-\vhe(t,\cdot)\|_{\lnorm 1(\L_t)}+\|\uhe(t,\cdot)-\vhe(t,\cdot)\|_{\lnorm 1(\mathfrak{R}_t)}\\
	&\leq & \oo \bigg(\|\uhe(0,\cdot)-\vhe(0,\cdot)\|_{\lnorm 1(0,\hat L)} + \int^{t}_{0} \Big(\big|b_1(\uhe(s,0+))-b_1(\vhe(s,0+))\big|\\
	& &\qquad \ \qquad + \big|b_2(\uhe(s,L-))-b_2(\vhe(s,L-))\big|\Big)ds +\e t \bigg),\quad \forall t\in [0,\hat \tau).
	\end{eqnarray*}
	
	Similarly, we have
	\begin{eqnarray*}
	&&\|\uhe(t,\cdot)-\vhe(t,\cdot)\|_{\lnorm 1(0,L)}\\
	&\leq & \oo \bigg(\|\uhe(j\hat\tau,\cdot)-\vhe(j\hat\tau,\cdot)\|_{\lnorm 1(0,L)} + \int^{t}_{j\hat\tau} \Big(\big|b_1(\uhe(s,0+))-b_1(\vhe(s,0+))\big|\\
	& &\qquad \qquad + \big|b_2(\uhe(s,L-))-b_2(\vhe(s,L-))\big|\Big)ds +\e t \bigg)
	\end{eqnarray*}
	for all $t\in [j\hat\tau,(j+1)\hat\tau)\cap(0,T),\ j\in \Z\cap(0,T/\hatt)$. 
	Then, by induction we can obtain estimate \eqref{est:ueh-s} of
	approximate stability for $\ueh$.

	\subsubsection{Proof of Proposition \ref{p:uniq}}
	By Proposition \ref{p:ueh-stab}, it is easy to see that for a fixed $h>0$, there exists $\delta>0$ so small that for some given $ s_0,s+s_0\in [0,\hatt)\ (s>0)$ and any given initial-boundary data $\iu:(0,L)\to \rn n,\  g_1:(s_0,s+s_0)\to \rn{n-m},\ g_2:(s_0,s+s_0)\to \rn{m}$ with $\lug<\delta$, the solution $\uh=\uh(t,x)$ to system \eqref{e:gs} associated with the initial-boundary condition (see Figure \ref{f:ts})
	\begin{equation*}
	\begin{cases}
	t=s:& u=\iu(x),\quad x\in (0,L),\\
	x=0:&b_1(u)=g_1(t),\quad t\in (s_0,s+s_0),\\
	x=L:&b_2(u)=g_2(t),\quad t\in (s_0,s+s_0),
	\end{cases}
	\end{equation*}	
	given by Proposition \ref{p:uhs}, restricted on the trapezoid domain 
	\begin{equation}\label{d:mflstar}
	\mfl^*(s,s_0):=\{ s_0<t<s+s_0,\ 0<x<L(t-\hatt)/\hatt  \},
	\end{equation}
	 coincides with the unique solution  to \eqref{e:uh}, as the limit of $\eh$-solutions to the one-sided initial-boundary value problem of system \eqref{e:hcl} associated with the initial-boundary condition (see Figure \ref{f:os})
	\begin{equation}
	\label{e:oib}
	\begin{cases}
	t=s_0:& u=\iu(x),\quad x\in \mfl_{s_0},\\
	x=0:&b_1(u)=g_1(t),\quad t\in (s_0,s+s_0).
	\end{cases}
	\end{equation}
		\begin{figure}[h!]
		\begin{minipage}[t]{0.42\linewidth} 
			\centering
			\includegraphics[width=0.65\linewidth]{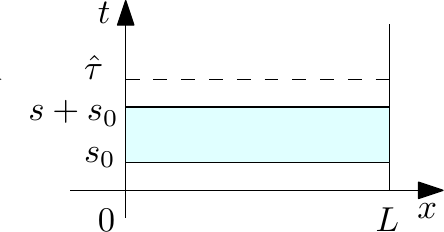}
			\caption{the initial-boundary value problem with two-sided boundary conditions}
			\label{f:ts}
		\end{minipage}%
		\qquad\qquad
		\begin{minipage}[t]{0.42\linewidth}
			\centering
			\includegraphics[width=0.65\linewidth]{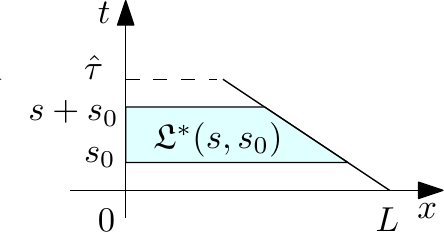}
			\caption{the initial value problem on $ \mfl^*(s,s_0) $ with one-sided boundary condition}
			\label{f:os}
		\end{minipage}
	\end{figure}
	
	Based on this, in what follows we construct a family of closed domain $ \md_t^h(\delta) \subset \L^1((0,L);\rn n)\times \L^1((0,T);\rn{n-m})$ with $ t\in [0,\hatt) $ for $ \delta$ and $ h $ sufficiently small, such that for each initial-boundary data $ (\iu,g_1) \in \md_0^h(\delta)$, the solution $ \uh=\uh(t,x) $ to system \eqref{e:uh} on $ \mfl $, provided by Proposition \ref{p:uhs} and associated with the initial-boundary data $ (\iu,g_1) $, 
	satisfies that $ \uh(t,\cdot)\in \md_t^h(\delta) $ for any $ t\in (0,\hatt) $, if $ \uh $ is extended to the domain $ \D_{\hatt} $ by taking $ \uh(t,x)\equiv 0 $ for all $ (t,x)\in \D_{\hatt}\setminus\mfl $. 
	
	For any given $ s\in (0,\hatt) $, let $ \iu:(0,L)\to \rn n $ and $ g:(s,\hatt)\to \rn{n-m} $ be two piecewise constant functions satisfying $\iu(x)\equiv 0$ for all $ x\in(0,L)\setminus \mfl_s  $ (recall notation \eqref{d:mflr}) with $\displaystyle \tv{\iu}{x\in\mfl_{s}}+\tv{g}{s<t<\hatt}+|b_1(\iu(0+))-g(0+)| $ sufficiently small. Let $ I $  be the set consisting of all the waves determined by all the $ h $-Riemann problems on the lattice $ \{jh \}_{j\in \Z\cap (0,\frac{2\hatt-s}{2\hatt h})} $ and by all the homogeneous initial (resp. initial-boundary) Riemann problem for the remaining discontinuities of $ \iu $ on $ \mfl_s $ (resp. the boundary $ x=0 $).
	Then we can define $ V_h^L,\ Q_h^L $ and $ \Upsilon_h^L $ by
	\begin{align*}
	&V_h^L(\iu,g)=\sum_{\alpha\in I} K^L_{\alpha}|\sa|+C_1\tv{g(t)}{s<t<\hatt},\\
	& Q_h^L(\iu,g)=\sum_{\substack{\alpha,\beta\in I\\ (\alpha,\beta)\in\mathcal{A}}} |\sa\sb|,\\
	& \Upsilon_h^L(\iu,g)=V_h^L(\iu,g)+C_2 Q_h^L(\iu,g)
	\end{align*}
	where the constants $ C_1 $ and $ C_2 $ are chosen as the same for $ V^L$ and $ \Upsilon^L $ used in Section \ref{ss:puhe}, respectively.
	The main difference between $ \Upsilon^L_h $ and $ \Upsilon^L $ is that $ \Upsilon^L_h $ can be defined for any piecewise constant functions with small total variations, not necessary to be $ \eh $-solutions.
	
	For any given $t\in (0,\hatt)$, 
	we can define 
	\begin{align*}
	\md_t^h(\delta):=\cl\Big\{& (u,g):\ u:(0,L)\to \rn{n}, \ g:(t,\hatt)\to
	\rn{n-m} \ \text{piecewise} \\
	&\  \text{ constant}, u(x)\equiv 0 \ \text{for all}\ x\in(0,L)\setminus \mflt, \ \text{and}\ \Upsilon_h^L(u,g)<\delta \Big \}
	\end{align*}
	depending on $ h,\ \delta $ and $ t $, where the notation ``cl" denotes the closure of the set. Therefore, for $\delta>0$ sufficiently small, there exists a unique process $P^h(t,t_0):\md_{t_0}^h(\delta)\to \md_{t+t_0}^h(\delta)$ for all $t_0,t\in (0,\hatt]$, such that for each $U_0=(\iu,g_1)\in\md_{t_0}^h(\delta)$ we have
	\[
	P^h(t,t_0)U_0=(E^h(t,t_0)U_0,\T(t,t_0)g_1),
	\]
	where $E^h(t,t_0)$ is the evolution operator whose trajectory $u^h(t,\cdot)=E^h(t,0)U_0$ 
	has the following property: restricted on $ \mfl^*(t,t_0) $ (recall notation \eqref{d:mflstar}), $ \uh=\uh(t,x) $ solves the one-sided initial-boundary value problem \eqref{e:uh} and \eqref{e:oib} (with $ s_0=t_0 $ and $ s=t $), obtained as the limit of $\eh$-solutions as $\e\to 0$ with fixed $h$. In particular, we have
	\begin{enumerate}[(a)]
		\item for all $t_0\in (0,\hatt]$ and $U\in \md_{t_0}^h(\delta)$, $P^h(0,t_0)U=U$;
		\item for all $U\in \md_{t_0}^{h}$, $P^h(t+s,t_0)U=P^h(t,t_0+s)\circ P^h(s,t_0)U$;
		\item for any given $U=(u,g^u)\in \md_{t_0}^h, V=(v,g^v)\in \md_{t_0'}^h$, we have
		\begin{align*}
		\|E^h(t,t_0)U-E^h(t',t_0')V\|_{\L^1(\mfl_{t+t_0})}\le& L \bigg\{\|u-v \|_{\L^1(\mfl_{t_0})}+|t-t'|\\
		&+|t_0-t_0'|+\int\limits_{t_0}^{t_0+t}|g_u(\tau)-g_v(\tau)|d\tau \bigg\}.
		\end{align*}
	\end{enumerate}
	
	Now we construct the family of domains $ \{\md_t\}_{t\in[0,\hatt)} $ in Proposition \ref{p:uniq} as follows. There exist $\delta'$ and $\delta''$ with $\delta'<\delta''$, such that
	\[
	\md_t^{h'}(\delta')\subseteq \md_t^{h''}(\delta''), \qquad \forall h',h''>0,\  t\in (0,\hatt].
	\]
	Noticing that $\md_{t_1}^h(\delta)\subseteq \md_{t_2}^h(\delta)$ for $ t_1<t_2 $, let
	\begin{equation*}
	\hat{\md}_t=\bigcap_{h>0} \md_t^h(\delta'), \quad  \hat{\md}_t'=\bigcap_{h>0}\cl \left\{  \bigcup_{k<h}  \md_t^k(\delta') \right\} 
	\end{equation*}
	and
	\begin{equation*}
	\check{\md}_{t}=\bigcap_{h>0}\md_t^h(\delta''), \quad \check{\md}_t'=\bigcap_{h>0}\cl \left\{  \bigcup_{k<h}  \md_t^k(\delta'') \right\}.
	\end{equation*}
	 It is easy to see that $\hat{\md}_t\subseteq \check{\md}_t$. Moreover, there exists a sequence of $\{ P^{h_i} \}$, such that for any given $U\in \check{\md}_t$, $P^{h_i}(s,t)$ converges as $ h_i\to 0 $ to a function $P(s,t)U$ with
	\begin{align}
	& P(s,t)U\in \hat{\md}_{t+s}',\quad \text{if}\  u\in \hat{\md}_{t},\\
	& P(s,t)U\in \check{\md}_{t+s}',\quad \text{if}\  u\in \check{\md}_{t}. \label{e:20}
	\end{align}
	Defining $ \md_t $ by
	\begin{align*}
	\md_t=\left\{ V\in \hat{\md}_t':\ \exists\ U_k\in \hat{\md}_0\ \text{and}\ t_k\downarrow t,\ \text{s.t.}\ V=\lim_{k\to \infty}P(t_k,0)U_k \right\},
	\end{align*}
	By the fact $ \md_{t}\subseteq \check{\md}_t $ and property (c) of $ P^h $ mentioned above, it is easy to check that (ii) and (iii) of Proposition \ref{p:uniq} hold. Then, (iv) can be deduced directly from (iii) and Proposition \ref{p:uhex}. By (ii) and \eqref{e:20} and passing to the limits in (a) and (b), we can show that (i) of Proposition \ref{p:uniq} holds for $ U\in \hat{\md_{t_0}} $ with $ t_0=0 $. It remains to check (i), that is, if $U\in \md_t$, then $P(s,t)U\in \md_{t+s}$. In fact, suppose $U\in \md_t$, then there exist $U_k \in \hat{\md}_0$ and $s_k\downarrow s$, such that $U=\displaystyle \lim_kP(s_k,0)U_{k}$. Thus we have $\displaystyle P(t,0)U= \lim_kP(t,s_k)P(s_k,0)U_k=\lim_{k}P(t+s_k,0)U_{k}$, which yields $P(t,0)U\in \md_{s+t}$. Therefore, we can take the limits in (a) and (b) and obtain that (i) holds for any given $ U\in \md_{t_0} $ with $ t_0 \in [0,\hatt)$.




%
%

\end{document}